\documentclass[12pt]{amsart}
\usepackage{amsmath,amssymb,amsbsy,amsfonts,latexsym,amsopn,amstext,cite,
                                               amsxtra,euscript,amscd,bm}
\usepackage{url}

\usepackage{mathrsfs}

\usepackage{color}
\usepackage[colorlinks,linkcolor=blue,anchorcolor=blue,citecolor=blue,backref=page]{hyperref}
\usepackage{color}
\usepackage{graphics,epsfig}
\usepackage{graphicx}
\usepackage{float}
\usepackage{epstopdf}
\hypersetup{breaklinks=true}

\usepackage[np]{numprint}
\npdecimalsign{\ensuremath{.}}

\usepackage{bibentry}

\usepackage[english]{babel}
\usepackage{mathtools}
\usepackage{todonotes}
\usepackage{url}
\usepackage[colorlinks,linkcolor=blue,anchorcolor=blue,citecolor=blue,backref=page]{hyperref}

\usepackage[norefs,nocites]{refcheck}

\DeclareMathOperator*\uplim{\overline{lim}}

\begin{document}

\newtheorem{theorem}{Theorem}
\newtheorem{lemma}[theorem]{Lemma}
\newtheorem{claim}[theorem]{Claim}
\newtheorem{cor}[theorem]{Corollary}
\newtheorem{conj}[theorem]{Conjecture}
\newtheorem{prop}[theorem]{Proposition}
\newtheorem{definition}[theorem]{Definition}
\newtheorem{question}[theorem]{Question}
\newtheorem{example}[theorem]{Example}
\newcommand{\hh}{{{\mathrm h}}}
\newtheorem{remark}[theorem]{Remark}

\numberwithin{equation}{section}
\numberwithin{theorem}{section}
\numberwithin{table}{section}
\numberwithin{figure}{section}

\def\sssum{\mathop{\sum\!\sum\!\sum}}
\def\ssum{\mathop{\sum\ldots \sum}}
\def\iint{\mathop{\int\ldots \int}}

\newcommand{\diam}{\operatorname{diam}}

\def\squareforqed{\hbox{\rlap{$\sqcap$}$\sqcup$}}
\def\qed{\ifmmode\squareforqed\else{\unskip\nobreak\hfil
\penalty50\hskip1em \nobreak\hfil\squareforqed
\parfillskip=0pt\finalhyphendemerits=0\endgraf}\fi}

\newfont{\teneufm}{eufm10}
\newfont{\seveneufm}{eufm7}
\newfont{\fiveeufm}{eufm5}
%
%
\newfam\eufmfam
     \textfont\eufmfam=\teneufm
\scriptfont\eufmfam=\seveneufm
     \scriptscriptfont\eufmfam=\fiveeufm
%
%
\def\frak#1{{\fam\eufmfam\relax#1}}

\newcommand{\bflambda}{{\boldsymbol{\lambda}}}
\newcommand{\bfmu}{{\boldsymbol{\mu}}}
\newcommand{\bfxi}{{\boldsymbol{\eta}}}
\newcommand{\bfrho}{{\boldsymbol{\rho}}}

\def\eps{\varepsilon}

\def\fK{\mathfrak K}
\def\fT{\mathfrak{T}}
\def\fL{\mathfrak L}
\def\fR{\mathfrak R}

\def\fA{{\mathfrak A}}
\def\fB{{\mathfrak B}}
\def\fC{{\mathfrak C}}
\def\fM{{\mathfrak M}}
\def\fS{{\mathfrak  S}}
\def\fU{{\mathfrak U}}

\def\T {\mathsf {T}}
\def\Tor{\mathsf{T}_d}
\def\Tore{\widetilde{\mathrm{T}}_{d} }

\def\sM {\mathsf {M}}

\def\ss{\mathsf {s}}

\def\Kmnd{\cK_d(m,n)}
\def\Kmnp{\cK_p(m,n)}
\def\Kmnq{\cK_q(m,n)}

\def \balpha{\bm{\alpha}}
\def \bbeta{\bm{\beta}}
\def \bgamma{\bm{\gamma}}
\def \bdelta{\bm{\delta}}
\def \bzeta{\bm{\zeta}}
\def \blambda{\bm{\lambda}}
\def \bchi{\bm{\chi}}
\def \bphi{\bm{\varphi}}
\def \bpsi{\bm{\psi}}
\def \bnu{\bm{\nu}}
\def \bomega{\bm{\omega}}

\def \bell{\bm{\ell}}

\def\eqref#1{(\ref{#1})}

\def\vec#1{\mathbf{#1}}

\newcommand{\abs}[1]{\left| #1 \right|}

\def\Zq{\mathbb{Z}_q}
\def\Zqx{\mathbb{Z}_q^*}
\def\Zd{\mathbb{Z}_d}
\def\Zdx{\mathbb{Z}_d^*}
\def\Zf{\mathbb{Z}_f}
\def\Zfx{\mathbb{Z}_f^*}
\def\Zp{\mathbb{Z}_p}
\def\Zpx{\mathbb{Z}_p^*}
\def\cM{\mathcal M}
\def\cE{\mathcal E}
\def\cH{\mathcal H}

\def\le{\leqslant}

\def\ge{\geqslant}

\def\sfB{\mathsf {B}}
\def\sfC{\mathsf {C}}
\def\L{\mathsf {L}}
\def\FF{\mathsf {F}}

\def\sE {\mathscr{E}}
\def\sS {\mathscr{S}}

\def\cA{{\mathcal A}}
\def\cB{{\mathcal B}}
\def\cC{{\mathcal C}}
\def\cD{{\mathcal D}}
\def\cE{{\mathcal E}}
\def\cF{{\mathcal F}}
\def\cG{{\mathcal G}}
\def\cH{{\mathcal H}}
\def\cI{{\mathcal I}}
\def\cJ{{\mathcal J}}
\def\cK{{\mathcal K}}
\def\cL{{\mathcal L}}
\def\cM{{\mathcal M}}
\def\cN{{\mathcal N}}
\def\cO{{\mathcal O}}
\def\cP{{\mathcal P}}
\def\cQ{{\mathcal Q}}
\def\cR{{\mathcal R}}
\def\cS{{\mathcal S}}
\def\cT{{\mathcal T}}
\def\cU{{\mathcal U}}
\def\cV{{\mathcal V}}
\def\cW{{\mathcal W}}
\def\cX{{\mathcal X}}
\def\cY{{\mathcal Y}}
\def\cZ{{\mathcal Z}}
\newcommand{\rmod}[1]{\: \mbox{mod} \: #1}

\def\cg{{\mathcal g}}

\def\vy{\mathbf y}
\def\vr{\mathbf r}
\def\vx{\mathbf x}
\def\va{\mathbf a}
\def\vb{\mathbf b}
\def\vc{\mathbf c}
\def\ve{\mathbf e}
\def\vh{\mathbf h}
\def\vk{\mathbf k}
\def\vm{\mathbf m}
\def\vz{\mathbf z}
\def\vu{\mathbf u}
\def\vv{\mathbf v}

\def\e{{\mathbf{\,e}}}
\def\ep{{\mathbf{\,e}}_p}
\def\eq{{\mathbf{\,e}}_q}

\def\Tr{{\mathrm{Tr}}}
\def\Nm{{\mathrm{Nm}}}

 \def\SS{{\mathbf{S}}}

\def\lcm{{\mathrm{lcm}}}

 \def\0{{\mathbf{0}}}

\def\({\left(}
\def\){\right)}
\def\l|{\left|}
\def\r|{\right|}
\def\fl#1{\left\lfloor#1\right\rfloor}
\def\rf#1{\left\lceil#1\right\rceil}
\def\sumstar#1{\mathop{\sum\vphantom|^{\!\!*}\,}_{#1}}

\def\mand{\qquad \mbox{and} \qquad}

\def\tblue#1{\begin{color}{blue}{{#1}}\end{color}}




\hyphenation{re-pub-lished}

\mathsurround=1pt

\def\bfdefault{b}

\def \F{{\mathbb F}}
\def \K{{\mathbb K}}
\def \N{{\mathbb N}}
\def \Z{{\mathbb Z}}
\def \P{{\mathbb P}}
\def \Q{{\mathbb Q}}
\def \R{{\mathbb R}}
\def \C{{\mathbb C}}
\def\Fp{\F_p}
\def \fp{\Fp^*}

 \def \xbar{\overline x}

\title{New bounds of Weyl sums}

 \author[C. Chen] {Changhao Chen}

\address{Department of Pure Mathematics, University of New South Wales,
Sydney, NSW 2052, Australia}
\email{changhao.chenm@gmail.com}

 \author[I. E. Shparlinski] {Igor E. Shparlinski}

\address{Department of Pure Mathematics, University of New South Wales,
Sydney, NSW 2052, Australia}
\email{igor.shparlinski@unsw.edu.au}

\begin{abstract}  
%

We  augment the method of  Wooley (2015) by some  new ideas and
in  a series of results, 
improve his metric  bounds  on the Weyl sums
and the discrepancy of fractional parts of  real  polynomials  with 
partially prescribed coefficients. 

We also  extend these  results and ideas  to  
principally new  and very  general settings
of arbitrary orthogonal projections of the vectors of the coefficients
$(u_1, \ldots , u_d)$ onto a lower dimensional subspace. This  new point of view has an additional 
advantage of yielding an upper bound on the Hausdorff dimension of sets of large Weyl sums. 
Among other technical innovations, we also introduce a ``self-improving'' approach, 
which leads to an infinite series of monotonically decreasing bounds, converging to our
final result. 
\end{abstract}


\maketitle

\tableofcontents

\section{Introduction}

\subsection{Background}
For an integer $d \geqslant 2$, let $\Tor = (\R/\Z)^d$ be the  $d$-dimensional unit torus. The exponential 
sums  
\begin{equation}
\label{eq:Sd}
S_d(\vu; N)=\sum_{n=1}^N \e(u_1n+\ldots+u_d n^d), \quad \vu = (u_1, \ldots, u_d) \in \Tor, 
\end{equation}
have been  introduced and estimated by Weyl~\cite{Weyl},  and thus are called the~\emph{Weyl sums}, where throughout the paper we denote 
$$
\e(x) = \exp(2\pi ix).
$$
By investigating the properties of the sums~\eqref{eq:Sd}, Weyl~\cite{Weyl}  established  the
  {\it uniformity of distribution modulo one\/}  
of  the sequence  
$$
u_1n+\ldots+u_d n^d, \qquad n\in \N ,  
$$
provided at least one of the coefficients $u_1, \ldots, u_d$  is irrational. 
The Weyl sums play crucial role in many other fundamental number theoretic problems. These include 
estimating the {\it zero-free region} of the {\it Riemann zeta-function} and thus obtaining good bounds in  the error term  in  the {\it prime number theorem}, see~\cite[Section~8.5]{IwKow},  
and  the {\it Waring problem}, see~\cite[Section~20.2]{IwKow} or~\cite{Vau} for a more detailed treatment. 
Further problems include bounds of very short character sums modulo highly composite numbers~\cite[Section~12.6]{IwKow}
and various problems from the uniformity of distribution theory and Diophantine approximations~\cite{Baker}.

However, despite  more than a century long history of estimating such sums, the behaviour of individual 
sums is not well understood. There have been several conjectures made about their behaviour 
and true order of magnitude of such sums depending on Diophantine properties of the coefficients 
$u_1, \ldots, u_d$; some have been ruled out,  some are still widely open even in the case of sums with 
monomials $un^d$,  
see~\cite{Brud,BD}.

The following bound is a direct implication of the current form of the  {\it Vinogradov mean value theorem} from~\cite{BDG, Wool2}
 and is 
 explicitly 
given in~\cite[Theorem~5]{Bourg}.   Let $\vu = (u_1, \ldots, u_d)  \in \Tor$ be such that for some $\nu$ with $2 \le \nu\le d $ and some positive integers $a$ and $q$ with $\gcd(a,q)=1$  we have
$$
\left| u_\nu - \frac{a}{q}\right| \le \frac{1}{q^2}. 
$$
Then for any $\varepsilon>0$ there exits a constant $C(\varepsilon)$ such that
\begin{equation}
\label{eq:individual}  
|S_d(\vu; N)| \le C(\varepsilon)N^{1+\varepsilon} \(q^{-1} + N^{-1} + qN^{-\nu}\)^{1/d(d-1)}. 
\end{equation}
It seems that the current bounds are expected to be far away from 
the true size of $S_d(\vu; N)$. We also remark that  as mentioned by Bourgain~\cite[Section~3]{Bourg}, for $d \le 6$ better results are known.

On the other hand, the behaviour of the average value of the Weyl sums has recently been fully 
unveiled in works of  Bourgain, Demeter and Guth~\cite{BDG} (for $d \geqslant 4$) 
and Wooley~\cite{Wool2} (for $d=3$)  (see also~\cite{Wool5}) in the best possible 
form  
\begin{equation}
\label{eq:MVT}
\int_{\Tor} |S_d(\vu; N)|^{2s(d)}d\vu \leqslant  N^{s(d)+o(1)}, \qquad N \to  \infty,
\end{equation}
of  the Vinogradov mean value theorem, 
where  for $q\in \R$ we denote 
\begin{equation}
\label{eq:sq}
s(q)=\frac{q(q+1)}{2}.
\end{equation}

Here we study a question which originates from the work of Flaminio and Forni~\cite{FlFo} 
and 
has also been studied in more detail by Wooley~\cite{Wool3}.
 Namely, here we seek results  
 which hold for {\it all\/} values of the components of  $\vu = (u_1, \ldots, u_d) \in \Tor$     
on some prescribed set of positions and  {\it almost all\/} values of the components on the remaining positions.  
Thus this question ``interpolates'' between 
{\it individual\/}  bounds and {\it bounds\/} involving some kind of averaging.   Wooley~\cite[Theorem~1.1]{Wool3} has shown that in this setting the individual bound in~\eqref{eq:individual} can be improved. 
 In this project we introduce several additional arguments and  make further improvements.

\subsection{Set-up and previous results}
Given a family $\bphi = \(\varphi_1, \ldots, \varphi_d\)\in \Z[T]^d$  of $d$ distinct 
nonconstant polynomials and a sequence of complex  weights $\vec{a} = (a_n)_{n=1}^\infty$, for $\vu=(u_1, \ldots, u_d)\in \Tor$ we  define the 
trigonometric polynomials 
\begin{equation}
\label{eq:Tphi}
T_{\va, \bphi}( \vu; N)=\sum_{n=1}^{N}a_n \e\(u_1 \varphi_1(n)+\ldots + u_d\varphi_d(n) \).
\end{equation} 

Furthermore, for $k=1, \ldots, d$, decomposing 
$$
\Tor = \T_k\times \T_{d-k}
$$ 
with $\T_k=[0,1)^{k}$ and $\T_{d-k}=[0,1)^{d-k}$. Given  $\vx\in \T_k$, $\vy\in\T_{d-k}$ we refine the notation~\eqref{eq:Tphi}
and  write
$$
T_{\va, \bphi}( \vx, \vy; N)= \sum_{n=1}^{N}a_n \e\(\sum_{j=1}^k x_j \varphi_j(n)+ \sum_{j=1}^{d-k}y_j\varphi_{k+j}(n) \).
$$ 
If $\va = \ve =  (1)_{n=1}^\infty$ (that is,  $a_n =1$ for each $n\in \N$) we just write 
$$
T_{\bphi}(\vx, \vy; N) = T_{\ve, \bphi}( \vx, \vy; N).
$$

For the classical case $a_n = 1$ for all $n\in \N$ and the polynomials
\begin{equation}
\label{eq:classical}
  \{\varphi_1(T),\ldots, \varphi_d(T)\}=\{T, \ldots, T^d\} 
\end{equation} 
satisfying some natural necessary conditions, 
the result of Wooley~\cite[Theorem~1.1]{Wool3} together with the modern knowledge towards the 
Vinogradov mean value theorem, see~\eqref{eq:MVT},  asserts that for almost all $\vx\in \T_k $ 
with respect to the $k$-dimensional Lebesgue measure on $\T_k$, one has 
\begin{equation}
\label{eq:wooley}
\sup_{\vy\in \T_{d-k}} |T_{\bphi}(\vu, \vy; N)| \le N^{\Gamma_\ast(\bphi,k) +o(1)}, \qquad N\to \infty, 
\end{equation}
where 
$$
\Gamma_\ast(\bphi,k) =\frac{1}{2} + \frac{2\sigma_k(\bphi)+d-k+1}{2d^2+4d-2k+2}
$$
and 
\begin{equation}
\label{eq:sigmak}
\sigma_k(\bphi)=\sum_{j=k+1}^d \deg \varphi_j . 
\end{equation}
We remark that the bound~\eqref{eq:wooley} is presented in a more explicit form than in~\cite[Theorem~1.1]{Wool3} 
as we have used the optimal result of Wooley~\cite[Theorem~1.1]{Wool5} for the parameter $u$ of~\cite[Theorem~1.1]{Wool3}. Furthermore  the results in~\cite[Theorem~1.1]{Wool3} have the restriction  that $k<d$, but our method  works for $k=d$  also. Naturally,  for the case $k=d$ we consider $\vx=\vu$ only and  remove the  variable $\vy$ from each statement for this special case.

Here  we use some new ideas to extend the method and results  of Wooley~\cite{Wool3} in serval directions. In particular,  we obtain an improvement of~\eqref{eq:wooley}. 

We note that it is also interesting to find 
a tight upper bound  for the almost all points $\vu\in \Tor$ for the classical  Weyl sums $S_d(\vu; N)$ given by~\eqref{eq:Sd}. 
In this direction the authors~\cite[Appendix~A]{ChSh} have shown  that for almost all $\vu\in \Tor$ one has 
\begin{equation}
\label{eq:1/2}
\left|S_d(\vu; N)\right| \le N^{1/2+o(1)}, \qquad N\to \infty.
\end{equation}
It is very natural to  conjecture that the exponent $1/2$ cannot be improved. 

Fedotov and Klopp~\cite[Theorem~0.1]{FK} have shown that the conjecture is true for    $d=2$. 
More precisely,  for any  non-decreasing sequence  $\{g(n)\}_{n=1}^{\infty}$ of positive numbers,  
 for almost all $\vu\in \mathbb{T}_2$ we have  
$$
\uplim_{N\to  \infty} \frac{  \left |S_2(\vu; N)\right |}{\sqrt{N} g(\ln  N)}<\infty \quad  \Longleftrightarrow \quad  \sum_{n=1}^{\infty} \frac{1}{ g(n)^{6}} <\infty.
$$  
We remark that  the  conjecture is still open for $d\ge 3$.

As in~\cite{Wool3},  we give applications of our bounds of  exponential sums to bounds on the {\it discrepancy\/}
(see Section~\ref{eq:Discr} for a definition) 
 of the sequence of fractional parts of polynomials.   However, we modify and improve the approach of Wooley~\cite{Wool3}  of passing from exponential sums to the  discrepancy  and obtain stronger results.  

\subsection{An overview of our results and tools} 
Here we obtain results of three different types:
\begin{itemize}
\item[(i)]  We study  the scenario of Wooley~\cite{Wool3}  when the vector   $\vu\in \Tor$ is split into two parts $\vx$ and $\vy$
formed by its components which is  related to the coordinate-wise projections of  $\vu\in \Tor$. 

\item[(ii)]  We introduce and study an apparently new problem related to arbitrary orthogonal  projections of   $\vu\in \Tor$.  
As an additional 
benefit, our results for arbitrary orthogonal  projections, combined 
with  the classical {\it Marstrand--Mattila projection theorem\/}, see~\cite[Chapter~5]{Mattila2015},  
lead to an upper bound on the Hausdorff dimension of sets or large Weyl sums, complementing our previous lower bounds~\cite{ChSh}. 

\item[(iii)]   As in~\cite{Wool3},   we study the uniform distribution of polynomials modulo one and obtain a bound for the discrepancy, which improves that of~\cite[Theorem~1.4]{Wool3}.  
\end{itemize}

We note that although our results improve those of~\cite{Wool3}, we see the main value of this work in 
new ways to combine several principal elements which have been used in the area for quite some time. Namely, we exploit the interplay between

\begin{itemize}
\item[(i)] the modern form of the Vinogradov mean value theorem, see, for example, Lemma~\ref{lem:Wooley};

\item[(ii)] the completion technique, see, for example,  Lemma~\ref{lem:control};

\item[(iii)] continuity of Weyl sums, see, for example,  Lemmas~\ref{lem:con gen}, \ref{lem:cont-gen} and~\ref{lem:cont-linear}, 
which in turn leads us to a new type of   ``self-improving''  results in  Lemma~\ref{lem:self}  and Corollary~\ref{cor:self}. 
\end{itemize}

As we have mentioned, as one of the applications of our results  we obtain  an upper bound for  the 
Hausdorff dimension of sets with large Weyl sums, see  Corollaries~\ref{cor: Ed k} and~\ref{cor: Ed d1}.
Other applications are given in Theorems~\ref{thm:weyl short} and~\ref{thm:discrep short} to a bound on short Weyl sums and to the distribution of fractional  parts of polynomials
over short intervals, respectively.

We hope that  similar combinations of these ideas may find several other applications.
 We also believe that the idea of studying arbitrary orthogonal 
projections and its applications to bounds of Hausdorff dimension has never been used in analytic number theory before this work.  

\section{Main results}

\subsection{Results for coordinate-wise projections of $\vu$:  a traditional point of view}

Throughout the paper, 
let $\bphi = \(\varphi_1, \ldots, \varphi_d\)\in \Z[T]^d$  be $d$ distinct nonconstant polynomials such that 
  the Wronskian  
\begin{equation}
\label{eq:Wronsk}
W(T;\bphi) = \det\(\varphi_i^{(j-1)}(T)\)_{i,j=1}^n
\end{equation} 
does not vanish identically and let  $\vec{a} = (a_n)_{n=1}^\infty$ 
be a sequence of complex  weights  with $a_n=n^{o(1)}$. 

We start with a very broad generalisation of~\eqref{eq:wooley}. 
We recall that 
$\sigma_k(\bphi)$ is given by~\eqref{eq:sigmak}. 

\begin{theorem}
\label{thm:general} 
Suppose that $\bphi \in \Z[T]^d$ is such that 
  the Wronskian  $W(T;\bphi)$ 
does not vanish identically, then for almost all $\vx\in \T_k$ one has
$$
\sup_{\vy\in \T_{d-k}} |T_{\va, \bphi}(\vx, \vy; N)| \le N^{\Gamma(\bphi,k) +o(1)},  \qquad N\to \infty, 
$$
where  
$$
\Gamma(\bphi,k) = \frac{1}{2}+\frac{2\sigma_k(\bphi)+d-k} {2d^2+4d-2k}.
$$ 
\end{theorem}

We remark that Theorem~\ref{thm:general} gives a nontrivial upper bound provided 
that $\sigma_k(\bphi)<s(d)$, where $s(d)$ is given by~\eqref{eq:sq}.

Furthermore for the classical choice of $\bphi$ as in~\eqref{eq:classical} we always have 
$\sigma_k(\bphi)<s(d)$.  
 Elementary calculations show that  
$$
\Gamma(\bphi, k)< \Gamma_\ast(\bphi,k),  \qquad k=1, \ldots, d.
$$  
Thus  Theorem~\ref{thm:general}  gives a direct improvement 
and generalisation of the bound~\eqref{eq:wooley},  which is due to Wooley~\cite[Theorem~1.1]{Wool3}.

We observe also  that  $\Gamma(\bphi, k)=1/2$ when $k=d$, and this gives the same bound as that of~\eqref{eq:1/2}  for more general polynomials $\bphi$. More precisely, we have the following.

\begin{cor}
Let $\bphi \in \Z[T]^d$, $d\ge 2$, be such that the Wronskian  $W(T;\bphi)$ 
does not vanish identically and let $a_n=n^{o(1)}$, then  for almost all $\vu\in \Tor$ one has  
$$
\left | \sum_{n=1}^N a_n \e(u_1 \varphi_1 (n) +\ldots +u_d \varphi_d(n))\right| \le N^{1/2+o(1)}, \qquad N\to \infty.
$$
\end{cor}

For some special cases of $\bphi \in \Z[T]^d$, we obtain a series of better bounds which in almost all cases are better than Theorem~\ref{thm:general} and thus give a further improvement of 
the  result  of Wooley~\cite[Theorem~1.1]{Wool3}.   The bounds are based on a new ``self-improving'' argument, see  Lemma~\ref{lem:self}
and  Corollary~\ref{cor:self} below. 

We consider the following three mutually exclusive possibilities:

\begin{itemize}
\item[\bf{A.}]  For some $k+1\le j\le d$ we have $\deg \varphi_j=1$, that is, with there is a linear polynomial attached to the vector $\vy$.
\item[\bf{B.}]  For some $1\le j\le k$ we have $\deg \varphi_j=1$, that is, with there is a linear polynomial attached to the vector $\vx$.
\item[\bf{C.}]  For all $1\le j\le d$ we have   $\deg \varphi_j \ge 2$, that is, $\bphi$ does not contain a linear polynomial. 
\end{itemize}

To reflect these there possibilities we denote new exponents,   replacing
$\Gamma(\bphi,k)$ by  $\Gamma_{YL} (\bphi,k)$,  $\Gamma_{XL} (\bphi,k)$ and  $\Gamma_{NL} (\bphi,k)$.

In fact our main result below Theorem~\ref{thm:A} handles only Case~{\bf{A}}. 
Then we reduce Cases~{\bf{B}} and~{\bf{C}} 
to   Case~{\bf{A}}.

Indeed, for Case~{\bf{B}},  assuming without loss of generality that $\deg \varphi_k =1$,  we simply write
 $T_{\va, \bphi}(\vx, \vy; N) =T_{\va, \bphi}(\check\vx, \hat \vy; N)$,  where we append $x_k$ to $\vy$ so that 
 $\hat \vy = (x_k, y_{k+1}, \ldots, y_d) \in \T_{d-k+1}$  which 
we estimate  for almost all $\check \vx = (x_1, \dots, x_{k-1}) \in \T_{k-1}$. 
That is, in Case~{\bf{B}}, for any $  \vx \in \T_k$ we use the inequality
$$
\sup_{\vy\in \T_{d-k}} |T_{\va, \bphi}(\vx, \vy; N)| \le \sup_{ \hat \vy \in \T_{d-k+1}} |T_{\va, \bphi}(\check\vx, \hat \vy; N)|. 
$$

To tackle  Case~{\bf{C}}, we simply replace $T_{\va, \bphi}(\vx, \vy; N)$
with $T_{\va, \widehat \bphi}(\vx, \widehat \vy; N)$, 
where we append $y_{d+1}$ to $\vy$ and $\varphi_{d+1}(T)  = T$
to $\bphi$, so that   $\widehat \vy = (y_{k+1}, \ldots, y_d, y_{d+1}) \in \T_{d-k+1}$  and 
$ \widehat{\bphi} = (\varphi_1, \ldots, \varphi_d, \varphi_{d+1}) \in \Z[T]^{d+1}$, 
which we estimate  for almost all $ \vx = (x_1, \dots, x_{k}) \in \T_{k}$. 
That is, in Case~{\bf{C}}, for any $  \vx \in \T_k$ we use the inequality
$$
\sup_{\vy\in \T_{d-k}} |T_{\va, \bphi}(\vx, \vy; N)| \le \sup_{ \widehat \vy \in \T_{d-k+1}} |T_{\va, \widehat \bphi}(\vx, \widehat \vy; N)|.
$$


More precisely, recalling the definitions~\eqref{eq:sq} and~\eqref{eq:sigmak} in Case~{\bf{A}} we have
the following bound.

\begin{theorem} 
\label{thm:A} 
Suppose that $\bphi \in \Z[T]^d$ is such that 
  the Wronskian  $W(T;\bphi)$ 
does not vanish identically and suppose that 
$$
 \min_{k<j\le d} \deg \varphi_j =1. 
$$
Then for almost all $\vx\in \T_k$ one has
$$
\sup_{\vy\in \T_{d-k}} |T_{\va, \bphi}(\vx, \vy; N)| \le N^{\Gamma_{YL}(\bphi,k) +o(1)},  \qquad N\to \infty, 
$$
where  
$$
\Gamma_{YL}(\bphi,k)= \frac{1}{2}+\frac{\sigma_k(\bphi)}{2s(d)}. 
$$
\end{theorem}

We remark that if $\sigma_k(\bphi)<d(d+1)/2$ then for each $k=1, \ldots, d-1$ one has  
$$
\Gamma_{YL}(\bphi,k)< \Gamma(\bphi,k).
$$
Moreover for the case $k=d$ we have  $\Gamma_{YL}(\bphi, d)=\Gamma(\bphi, d)=1/2$.
Thus Theorem~\ref{thm:A} improves Theorem~\ref{thm:general} if  there is a linear polynomial attached to the vector $\vy$.

As we have described in the above, for Cases~{\bf B} and~{\bf C}, from Theorem~\ref{thm:A} we obtain the following two estimates:

\begin{cor}
\label{cor:B} 
Suppose that $\bphi \in \Z[T]^d$ is such that 
  the Wronskian  $W(T;\bphi)$ 
does not vanish identically and suppose that $k\ge 2$ and 
$$
 \min_{1\le j\le k} \deg \varphi_j =1. 
$$
Then for almost all $\vx\in \T_k$ one has
$$
\sup_{\vy\in \T_{d-k}} |T_{\va, \bphi}(\vx, \vy; N)| \le N^{\Gamma_{XL}(\bphi,k) +o(1)},  \qquad N\to \infty, 
$$
where  
$$
\Gamma_{XL}(\bphi,k)= \frac{1}{2}+\frac{\sigma_k(\bphi)+1}{2s(d)}. 
$$
\end{cor}

\begin{cor} 
\label{cor:C} 
Suppose that $\bphi \in \Z[T]^d$ is such that 
  the Wronskian  $W(T;\widehat \bphi)$ of $  \widehat{\bphi} = (\varphi_1, \ldots, \varphi_d, \varphi_{d+1}) \in \Z[T]^{d+1}$ with $\varphi_{d+1}(T)=T$ 
does not vanish identically  and suppose that  
$$
 \min_{j=1, \ldots, d} \deg \varphi_j \ge 2. 
$$
Then for almost all $\vx\in \T_k$ one has
$$
\sup_{\vy\in \T_{d-k}} |T_{\va, \bphi}(\vx, \vy; N)| \le N^{\Gamma_{NL}(\bphi,k) +o(1)},  \qquad N\to \infty, 
$$
where  
$$
\Gamma_{NL}(\bphi,k)= \frac{1}{2}+\frac{\sigma_k(\bphi)+1}{2s(d+1)}. 
$$
\end{cor}



As yet another  application of Theorem~\ref{thm:A}  we derive the following bounds for the short sums. 
For $M\in \Z$, we consider Weyl sums over short intervals  
$$
S_d(\vu; M,N) = \sum_{n=M+1}^{M+N} \e(u_1n+\ldots+u_d n^d). 
$$

\begin{theorem}
\label{thm:weyl short}
For almost all $x_d\in  [0,1]$, one has 
$$
\sup_{(y_1, \ldots, y_{d-1})\in \T_{d-1}} \sup_{M \in \Z} \left |S_d(\vu; M,N)  \right | \le N^{1-1/(d+1) +o(1)}, \qquad N\to \infty, 
$$
where $\vu = (y_{1}, \ldots, y_{d-1}, x_d)$. 
\end{theorem}

From Theorem~\ref{thm:weyl short} we immediately derive that for almost all $\vu\in \Tor$ one has  
$$
\sup_{M \in \Z} \left |S_d(\vu; M,N)  \right | \le N^{1-1/(d+1)  +o(1)}, \qquad N\to \infty.
$$

Note that using the bound~\eqref{eq:individual} 
 and a similar  observation about the leading  
coefficient of shifted polynomials,   one obtains a version of  Theorem~\ref{thm:weyl short}  with the exponent 
$$
1- 1/d(d-1)  
 >  1-1/(d+1), \qquad d \ge 3.
$$

\begin{remark}
\label{rem:compare}
The bounds of Theorem~\ref{thm:A} and Corollaries~\ref{cor:B} and~\ref{cor:C}
are typically stronger than that of  Theorem~\ref{thm:general}. However in the case when  $\min_{j=1, \ldots, d} \deg \varphi_j \ge 2$
but  $W(T;\widehat \bphi)= 0$  this is the only result at our disposal.  
\end{remark}

\subsection{Results for arbitrary orthogonal projections of $\vu$:  a new point of view}

We now consider other projections which seems to be a new scenario which has not been
studied in the literature prior to this work.  
 
We need to introduce some notation first. 

Let $\cG(d, k)$ denote the collections of all the $k$-dimensional  linear subspaces of $\R^d$. For $\cV\in \cG(d, k)$,  let $\pi_\cV:~\R^{d} \to \cV$ denote the orthogonal projection onto $\cV$.
For  $0<\alpha<1$, we consider the set
$$
\cE_{\va, \bphi,\alpha}=\{\vu\in \Tor:~|T_{\va, \bphi }( \vu; N)|\ge N^{\alpha} \text{ for infinity many }N\in \N\}.
$$
We also use $\lambda\(\cS\)$ to denote the  Lebesgue measure of $\cS \subseteq \Tor$ (and also for sets in 
other spaces). 

We are interested in the following apparently new  point of view:

\begin{question}
\label{quest:proj}  Given $\bphi \in \Z[T]^d$, for what $\alpha$ 
we have $\lambda (\pi_\cV (\cE_{\va, \bphi,\alpha}))=0$ for all $\cV\in \cG(d, k)$? 
\end{question}

We now see that Theorem~\ref{thm:general} implies that for $\bphi \in \Z[T]^d$ is as in Theorem~\ref{thm:general}
and $a_n=n^{o(1)}$,  for any   
$$
\alpha>\frac{1}{2}+\frac{2\sigma_k(\bphi)+d-k}{2d^2+4d-2k},
$$
we have 
$$
\lambda (\pi_{d, k} (\cE_{\va, \bphi,\alpha}))=0,
$$
where $\pi_{d, k}$ is the orthogonal projection of $\Tor$ onto $\T_k$, that is, 
\begin{equation} 
\label{eq:pidk}
\pi_{d, k}:~(u_1, \ldots, u_d) \to (u_1, \ldots, u_k).
\end{equation}

For the degree sequence  $\deg \varphi_1, \ldots, \deg \varphi_d$ we denote them as 
\begin{equation} 
\label{eq:degree}
r_1\le \ldots\le r_d,
\end{equation}
and define 
\begin{equation}
\label{eq:wsigma}
\widetilde{ \sigma}_k(\bphi)=\sum_{i=k+1}^d r_i.
\end{equation}

We remark that  the following  result is  similar to the result of Theorem~\ref{thm:general},  with the change of $\widetilde{ \sigma}_k(\bphi)$ only.

\begin{theorem}
\label{thm:B} 
Suppose that $\bphi \in \Z[T]^d$ is such that 
  the Wronskian  $W(T;\bphi)$ 
does not vanish identically and $\widetilde{ \sigma}_k(\bphi)<d(d+1)/2$, then for any $\cV\in \cG(d, k)$ one has
$$
\lambda (\pi_\cV (\cE_{\va, \bphi,\alpha}))=0
$$
provided that $\alpha> \widetilde{\Gamma}(\bphi, k)$
where 
$$
\widetilde{\Gamma}(\bphi, k)=\frac{1}{2} + \frac{2 \widetilde{\sigma}_k(\bphi)+d-k}{2d^2+4d-2k}.
$$
\end{theorem}

We now consider Question~\ref{quest:proj} in  the classical case~\eqref{eq:classical} and the sums~\eqref{eq:Sd}. 
That is, we study the  following set
$$
\cE_{d, \alpha}=\{\vu \in \Tor:~|S_d(\vu; N)|\ge N^{\alpha} \text{ for infinity many } N\in \N\},
$$
which we define for  $0<\alpha<1$ and  integer $d\ge 2$. Note that in this setting we have  $\widetilde{\sigma}_k(\bphi)=(d+k+1)(d-k)/2$.

\begin{cor}
\label{cor:weyl-p}
For any $\cV\in \cG(d, k)$ one has 
$$
\lambda\(\pi_\cV \(\cE_{d, \alpha}\)\)=0
$$
provided that $\alpha>\widetilde{\Gamma}_{d,k} $
where
$$
\widetilde{\Gamma}_{d,k} = \frac{1}{2} + \frac{(d-k)(d+k+2)}{2d^2+4d-2k}.
$$
\end{cor}

We remark that the orthogonal projection of sets is a fundamental topic in fractal geometry and geometric measure theory. Recall the classical {\it Marstrand--Mattila projection theorem\/}: Let $\cA \subseteq \R^{d}$, $d\geq2,$ be a Borel set with Hausdorff dimension $s$,  see~\cite[Chapter~5]{Mattila2015}  for more details and related definitions.  Then we have:

\begin{itemize}
\item  {\it Dimension part:} If $s\leq k$, then the orthogonal projection of $\cA$ onto almost all $k$-dimensional subspaces has Hausdorff dimension $s$.

\item {\it Measure part:} If $s>k$, then the orthogonal  projection of $\cA$ onto almost all $k$-dimensional subspaces has positive $k$-dimensional Lebesgue measure. 
\end{itemize}

From the Marstrand--Mattila projection theorem and Corollary~\ref{cor:weyl-p} we obtain the following results. For $\cA \subseteq \R^d$ we use $\dim \cA$ to denote the Hausdorff dimension of $\cA$.

\begin{cor}
\label{cor: Ed k}
Let   $k, d$ be two integers with $1\le k<d$ and $d\ge 2$. Then $\dim \cE_{d, \alpha}\le k$ 
for any  
$$
\alpha> \frac{1}{2}+\frac{(d-k)(d+k+2)}{2d^2+4d-2k}.
$$
\end{cor}

In particular,  taking $k=d-1$ we obtain  
\begin{cor}
\label{cor: Ed d1}
For any integer $d\ge 2$ one has $\dim \cE_{d, \alpha}\le d-1$ for any 
$$
\alpha> \frac{1}{2}+\frac{2d+1}{2d^2+2d+2}.
$$ 
\end{cor}

We note that the authors~\cite{ChSh2}  showed that  for any $\alpha \in (1/2, 1)$ one has  
\begin{equation}
\label{eq:upperbound}
\dim \cE_{d, \alpha}\le u(d, \alpha)
\end{equation}
 with some explicit function $u(d, \alpha)<d$. Moreover the function 
 $$
 u(d, \alpha)\rightarrow 0 \quad \text{as} \quad \alpha\rightarrow 1.
 $$
However the exact comparison between  the bound $u(d, \alpha)$ and that of  Corollary~\ref{cor: Ed k} is not immediately 
 obvious.

We remark that the authors~\cite{ChSh} have obtained a  lower bound of the Hausdorff dimension of $\cE_{d, \alpha}$. Among other things, it is shown in~\cite{ChSh} that for any $d\ge 2$ and $\alpha\in (0,1)$ one has 
$$
\dim \cE_{d, \alpha}\ge \xi(d, \alpha)
$$
with some explicit function $\xi(d, \alpha)>0$. As a counterpart to~\eqref{eq:upperbound}, we remark 
that we expect $\dim \cE_{d, \alpha} = d$ for $\alpha\in (0,1/2)$, see also~\cite{ChSh, ChSh3}. On the other hand, we do not have any 
plausible conjecture about the exact behaviour of  $\dim \cE_{d, \alpha}$ for $\alpha\in [1/2, 1)$.

\begin{remark} 
In principle,  one can obtain various analogues
 of Theorem~\ref{thm:A} and  Corollaries~\ref{cor:B} and~\ref{cor:C} for the arbitrary projections. However they require imposing some additional
 (and rather cluttered)  restrictions 
 on linear combinations of components of $\bphi$.  We omit these similar but more involuted arguments for this setting. 
\end{remark}

\subsection{Uniform distribution modulo one} 
\label{sec:u.d.}
Let $\xi_n$, $n\in \N$,  be a sequence in $[0,1)$. The {\it discrepancy\/}  of this sequence at length $N$ is defined as 
\begin{equation}
\label{eq:Discr}
D_N = \sup_{0\le a<b\le 1} \left |  \#\{1\le n\le N:~\xi_n\in (a, b)\} -(b-a) N \right |.
\end{equation} 
We note that sometimes in the literature the scaled quantity $N^{-1}D_N $ is   called the discrepancy, but 
since our argument looks cleaner with the definition~\eqref{eq:Discr}, we adopt it here. 

For $\vx\in \T_k$, $\vy\in \T_{d-k}$ we consider the  sequence 
$$
\sum_{j=1}^{k} x_j \varphi_j(n) +\sum_{j=1}^{d-k} y_j \varphi_{k+j}(n), \qquad n\in \N,
$$ 
and for each $N$ we denote by  $D_{\bphi}(\vx, \vy; N)$ the corresponding discrepancy of its fractional parts.

Wooley~\cite[Theorem~1.4]{Wool3} has proved that ($d\ge 3$) for  almost all $\vx\in \T_k$ with $1\le k\le d-1$  
one has 
$$
\sup_{\vy\in \T_{d-k}}D_{\bphi}(\vx, \vy; N) \leq  N^{\gamma_\ast(\bphi, k)+o(1)},  \qquad N \to  \infty,
$$
where 
$$
\gamma_\ast(\bphi, k)=\frac{1}{2}+\frac{d-k+2\sigma_k(\bphi)+2}{2d^2+4d-2k+4}.
$$
We improve this bound as follows.

\begin{theorem}
\label{thm:general-D} Suppose that $\bphi \in \Z[T]^d$ is such that 
  the Wronskian  $W(T;\bphi)$ 
does not vanish identically and $\sigma_k(\bphi)<d(d+1)/2$, then  for almost all $\vx\in \T_k$  one has   
$$
\sup_{\vy\in \T_{d-k}}D_{\bphi}(\vx, \vy; N) \leq  N^{\gamma(\bphi, k)+o(1)} , \qquad   N \to \infty,
$$
where 
$$
\gamma(\bphi, k)=\frac{1}{2}+\frac{d-k+2\sigma_k(\bphi)+1}{2d^2+4d-2k+2}.
$$
\end{theorem}

For the classical choice of $\bphi$ as in~\eqref{eq:classical} we always have $\sigma_k(\bphi)<s(d)$, where $s(d)$ is given by~\eqref{eq:sq},   
and   elementary calculations show that  
$$
\gamma(\bphi, k)<\gamma_\ast(\bphi, k), \qquad k=1, \ldots, d.
$$ 
Thus, as before with Theorem~\ref{thm:general}, we see that  Theorem~\ref{thm:general-D}  gives a direct improvement and 
generalisation of the result of Wooley~\cite[Theorem~1.4]{Wool3}.  

 \begin{remark}
\label{rem:lin improve}  
It is natural to try to obtain analogues of the  bounds of  exponential sums of Theorem~\ref{thm:A} and Corollaries~\ref{cor:B} and~\ref{cor:C} for 
the discrepancy.  However our main tool, the {\it Erd\H{o}s--Tur\'{a}n inequality\/}, see Lemma~\ref{lem:ET} below, involves a growing with $N$
family of exponential sums of length $N$. So one needs some additional ideas to adjust our argument to this case. 
\end{remark}

From Theorem~\ref{thm:general-D}  we derive a bound on the discrepancy of real polynomials over short intervals.
More precisely, we now given an upper bound on $D_d(\vu;M, N)$  which  denotes the discrepancy  of the sequence  of fractional parts
$$
\{u_1n+\ldots+u_d n^d\}, \qquad  n=M+1, \ldots, M+N.
$$

\begin{theorem} 
\label{thm:discrep short}
For almost all $x_d\in  [0,1]$, one has
$$
\sup_{(y_1, \ldots, y_{d-1})\in \T_{d-1}} \sup_{M \in \Z} D_d(\vu; M,N)  \le N^{1-1/(d+2) +o(1)}, \qquad N\to \infty, 
$$
where $\vu = (y_{1}, \ldots, y_{d-1}, x_d)$.
\end{theorem}

From Theorem~\ref{thm:discrep short}  we obtain that for almost all $\vu\in \Tor$ one has   
$$
\sup_{M \in \Z}  D_d(\vu; M,N)  \le N^{1-1/(d+2)+o(1)}, \qquad N\to \infty.
$$

Finally, for $D_d(\vu; N)$ we claim that by combining~\cite[Theorem~5.13]{Harm} with some additional arguments,  
one can show that for  almost all $\vu\in \Tor$ with $d\ge 2$  one has the following stronger bound,
\begin{equation}
\label{eq:Du}
D_d(\vu; N)\le N^{1/2} (\log N)^{3/2 + o(1)}, \qquad N \to  \infty.
\end{equation}
We give a proof at Section~\ref{sec:comments}. Furthermore we conjecture that  this upper bound is the best possible except for a logarithm factor.   We   remark  that this is true for $d=2$ which follows by applying a result of Fedotov and Klopp~\cite[Theorem~0.1]{FK}  
and the  Koksma inequality~\cite[Theorem~5.4]{Harm}. However the conjecture is still open when $d\ge 3$. 

\section{Preliminaries} 
 
\subsection{Notation and conventions}

Throughout the paper, the notation $U = O(V)$, 
$U \ll V$ and $ V\gg U$  are equivalent to $|U|\leqslant c V$ for some positive constant $c$, 
which throughout the paper may depend on the degree $d$ and occasionally on the small real positive 
parameter $\varepsilon$. 

For any quantity $V> 1$ we write $U = V^{o(1)}$ (as $V \to \infty$) to indicate a function of $V$ which 
satisfies $|U| \le V^\eps$ for any $\eps> 0$, provided $V$ is large enough. One additional advantage 
of using $V^{o(1)}$ is that it absorbs $\log V$ and other similar quantities without changing  the whole 
expression.

 We use $\# \cS$ to denote the cardinality of a finite set $\cS$.

 We always identify $\Tor$ with half-open unit cube $[0, 1)^d$, in particular we
 naturally associate the Euclidean norm  $\|\vx\|$ with points $\vx \in \Tor$. 
 
We say that some property holds for almost all $\vx \in [0,1)^k$ if it holds for a set 
 $\cX \subseteq [0,1)^k$ of $k$-dimensional  Lebesgue measure  $\lambda(\cX) = 1$.

We always assume that  $\bphi \in \Z[T]^d$  consists of polynomials $\varphi_j$ of 
degrees
\begin{equation}
\label{eq:deg phi}
\deg \varphi_j = e_j, \qquad j =1, \ldots, d.
\end{equation}

\subsection{Generalised mean value theorems}

For the classical case of the Weyl sums $S_d(\vu; N)$  as in~\eqref{eq:Sd}, the {\it Parseval identity\/}   gives
$$
\int_{\Tor} |S_d(\vu; N)|^{2} d \vu= N.
$$
Furthermore, we have  the  Vinogradov mean value theorem, in the optimal  form~\eqref{eq:MVT}.

We use the following a general form due to Wooley~\cite[Theorem~1.1]{Wool5}, which extends the bound~\eqref{eq:MVT}
to the sums $T_{\va, \bphi}( \vu; N)$.

We also recall that for functions  $\bpsi = \(\psi_1, \ldots, \psi_d\) \in \Z[T]^d$ their  Wronskian $W(T;\bpsi) $   is 
defined in~\eqref{eq:Wronsk}.  

\begin{lemma}   
\label{lem:Wooley}
For any  a family $\bphi \in \Z[T]^d$  of $d$  polynomials 
such that  the Wronskian  $W(T;\bphi)$ 
does not vanish identically,  any sequence of complex  weights $\vec{a} = (a_n)_{n=1}^\infty$, and 
any integer $N\ge 1$,   we have the upper bound  
$$
\int_{\Tor} |T_{\va, \bphi}( \vu; N)|^{2\sigma } d\vu\le N^{o(1)} \(\sum_{n=1}^{N}|a_n|^2\)^{\sigma}
$$
for any real positive $\sigma \le s(d)$, where $s(d)$ is given by~\eqref{eq:sq}. 
\end{lemma}

\subsection{The completion technique}

We remark that the completion technique has many applications in analytic number theory. We show the following version for the later application.

\begin{lemma}\label{lem:control} 
For $\vu\in \Tor$ and $1\le M\le N$ we have 
$$
T_{\va, \bphi}( \vu; M) \ll W_{\va, \bphi}( \vu; N),
$$
where 
\begin{align*}
W_{\va, \bphi}&( \vu; N)\\
&=  \sum_{h=-N}^{N} \frac{1}{|h|+1} \left| \sum_{n=1}^{N} a_n  \e\(h n/N+u_1 \varphi_1(n)+\ldots + u_d\varphi_d(n) \) \right|.
\end{align*}
\end{lemma}

\begin{proof} For $\vu\in \Tor$ and $n\in \N$ denote
$$
f(n)=u_1 \varphi_1(n)+\ldots + u_d\varphi_d(n).
$$

Observe that by the orthogonality 
$$
\frac{1}{N}\sum_{h=1}^{N}\sum_{k=1}^{M} \e\left(h(n-k)/N\right)=
  \begin{cases}
   1 & n=1, \ldots, M, \\
  0 & \text{otherwise}.
  \end{cases}
$$
We also note that for $1 \le h,M \le N$ we have
$$
  \sum_{k=1}^{M}\e\left(hk/N\right )  \ll \frac{N}{\min\{h, N+1 - h\}},
$$
see~\cite[Equation~(8.6)]{IwKow}.  
It follows that   
\begin{align*}
T_{\va, \bphi}(\vu; M)
&=\sum_{n=1}^{N}a_n\e(f(n)) \frac{1}{N}\sum_{h=1}^{N}\sum_{k=1}^{M} \e\left(h(n-k)/N\right)\\
&=\frac{1}{N} \sum_{h=1}^{N}  \sum_{k=1}^{M}\e\left(-hk/N\right ) \sum_{n=1}^{N} a_n\e\left (hn/N+f(n) \right) \\
&\ll \sum_{h=1}^{N} \frac{1}{\min\{h, N+1 - h\}} \left |\sum_{n=1}^{N}a_n \e\left (hn/N+f(n) \right) \right |\\
&\ll  \sum_{h=-N}^{N} \frac{1}{|h|+1} \left |\sum_{n=1}^{N}a_n \e\(hn/N + f(n)\) \right |,
\end{align*}
which finishes the proof.
\end{proof}

For $\vx\in [0,1)^k$, $\vy\in [0,1)^{d-k}$, by Lemma~\ref{lem:control} we also have 
$$
T_{\va, \bphi}( \vx, \vy; N)\ll W_{\va, \bphi}( \vx, \vy; N),
$$
where 
\begin{align*}
& W_{\va, \bphi}( \vx, \vy; N) \\
&  \quad =  \sum_{h=-N}^{N} \frac{1}{|h|+1} \left| \sum_{n=1}^{N} a_n  \e\(h n/N + \sum_{j=1}^k x_j \varphi_j(n)+ \sum_{j=1}^{d-k}y_j\varphi_{k+j}(n) \) \right|.
\end{align*}

Note that for any $N$ there exists a sequence $b_N(n)$ such that 
$$
b_N(n) \ll \log N, \qquad n=1, \ldots, N, 
$$ 
and $W_{\va, \bphi}( \vu; N)$ can be written as 
\begin{equation}
\label{eq:W}
W_{\va, \bphi}( \vu; N)=\sum_{n=1}^N a_n b_N(n)  \e(u_1 \varphi_1(n)+\ldots + u_d\varphi_d(n)).
\end{equation} 
Indeed, since each inner sums in $W_{\va, \bphi}( \vu; N)$ depends only on $h$
(for a fixed $N$)  we clearly can write
$$
W_{\va, \bphi}( \vu; N)= \sum_{h=-N}^{N} \frac{\vartheta_h}{|h|+1}   \sum_{n=1}^{N} a_n  \e\(h n/N+u_1 \varphi_1(n)+\ldots + u_d\varphi_d(n) \) 
$$                           
for some complex $\vartheta_h$ on the unit circle. Hence we can take
$$
b_N(n) =  \sum_{h=-N}^{N} \frac{\vartheta_h}{|h|+1}  \e(h n/N) \ll \log N 
$$ 
in~\eqref{eq:W}. 
Combining~\eqref{eq:W}  with Lemma~\ref{lem:Wooley} we obtain the following.
\begin{cor}   
\label{cor:MVT-W}
Let $\bphi \in \Z[T]^d$  
such that  the Wronskian  $W(T;\bphi)$ 
does not vanish identically and  $a_n=n^{o(1)}$, then  we have   
$$
\int_{\Tor} |W_{\va, \bphi}( \vu; N)|^{2 s(d)} d\vu\le N^{s(d)+o(1)}.
$$
\end{cor}


\subsection{Continuity of exponential sums} 

We start with the following general statement which could be of independent interest.

\begin{lemma}  
\label{lem:con gen} Let integer $N \ge 1$ and a vector $\vu=(u_1, \ldots, u_d) \in \Tor$ be such that 
for any $1\le M \le N$ we have  
$$
T_{\va, \bphi}( \vu; M)\ll    M^{\rho} N^{o(1)}
$$
as $N\to \infty$, for some real   $\rho\ge 0$.  Then for any positive  $\omega = O(1)$  and  $\vv=(v_1, \ldots, v_d) \in \Tor$ 
with 
$$
u_i \le v_i< u_i+ \omega N^{-e_i} 
$$
 if $\varphi_i(n)>0$  for all large enough n, and
$$
u_i -\omega N^{-e_i} < v_i \le u_i,
$$
  if $\varphi_i(n)<0$ for all large enough $ n$,  
we obtain  
$$
T_{\va, \bphi}( \vv; N)- T_{\va, \bphi}( \vu; N) \ll   \omega N^{\rho+o(1)}, 
$$
where the implied constant is absolute.  
\end{lemma}

\begin{proof}
We first remark that the condition $\varphi_i(n)>0$  for all 
sufficiently large $n$ is equivalent to that  the polynomial $\varphi_i(n)$ is eventually an increasing function, 
which is  used below when we apply the partial sum formula.

Furthermore we remark that the  choice of $\vv=(v_1, \ldots, v_d)$ is to guarantee   the 
``non-negativity condition"
\begin{equation}
\label{eq:non-negative}
(v_i -u_i) \varphi_i(n)\ge 0, \quad i=1, \ldots, d,
\end{equation}
for all large enough $n$. 

Let  $\delta_i =  v_i-u_i$,  $i =1, \ldots, d$. 
For each $n\in \N$ we have 
\begin{align*}
&\e\(v_1 \varphi_1(n)+\ldots + v_d\varphi_d(n) \)\\
&\qquad  =\e\(u_1 \varphi_1(n)+\ldots + u_d\varphi_d(n) \)\e(\delta_1 \varphi_1(n)+\ldots + \delta_d \varphi_d(n))\\
&\qquad =\e\(u_1 \varphi_1(n)+\ldots + u_d\varphi_d(n) \) \\
& \qquad \qquad \qquad\quad \times \sum_{k=0}^{\infty} \frac{(2\pi i (\delta_1 \varphi_1(n)+\ldots + \delta_d \varphi_d(n))^{k}}{k!}.
\end{align*}

It follows that 
\begin{equation} \label{eq:diff}
\begin{split}
T_{\va, \bphi}( \vv; N)&- T_{\va, \bphi}( \vu; N)\\
&   = \sum_{k=1}^{\infty} \sum_{n=1}^{N} a_n \e\(u_1 \varphi_1(n)+\ldots + u_d\varphi_d(n) \) \\
& \qquad \qquad \qquad 
\times \frac{(2\pi i  (\delta_1\varphi_1(n)+\ldots + \delta_d \varphi_d(n))^{k}}{k!}.
\end{split} 
\end{equation}
For each $k\in \N$ we now turn to the estimate 
$$
\Sigma_k =\sum_{n=1}^{N} a_n\e\(\delta_1\varphi_1(n)+\ldots + \delta_d \varphi_d(n)\)  \xi_{n}^{k}, 
$$
where $\xi_{n}=\delta_1\varphi_1(n)+\ldots + \delta_d \varphi_d(n)$.   Applying partial sum formula we derive 
$$
\Sigma_k=  \Sigma_{k,1} + \Sigma_{k,2}
$$
where 
$$
\Sigma_{k,1} = T_{\va, \bphi}( \vu; N)\ \xi_{N}^k  \mand  \Sigma_{k,2}
 = \sum_{M=1}^{N-1} T_{\va, \bphi}( \vu; M)\(\xi_{M}^k -\xi_{M+1}^k\).
$$
By our assumption,  we obtain 
\begin{equation} \label{eq:sigma1}
\Sigma_{k,1}\ll  N^{\rho+o(1)}  (\delta_1N^{e_1}+\ldots + \delta_dN^{e_d})^k \le (d \omega)^k  N^{\rho+o(1)}.
\end{equation}
Observe that   there exists a constant $M_0$ (which depends on $\bphi$ only) 
 such that  the sequence $\xi_{M}$ is monotonically non-decreasing for all $M\ge M_0$.  It follows that 
\begin{equation}\label{eq:sigma2}
\begin{split}
 \Sigma_{k,2}   &\ll  \sum_{M=1}^{N-1}   M^{\rho}N^{o(1)} \left| \xi_{M}^k -\xi_{M+1}^k\right|\\
 &\ll  N^{\rho+o(1)} \left(\sum_{M=1}^{M_0} | \xi_{M}^k -\xi_{M+1}^k| + \sum_{M=M_0+1}^{N-1} (\xi_{M+1}^k -\xi_{M}^k )  \right)\\
 &\ll N^{\rho+o(1)} (M_0 (d \omega)^k+ (d \omega)^k)\\
 &\ll N^{\rho+o(1)}  (d \omega)^k.
\end{split}
\end{equation}

We see from~\eqref{eq:sigma1} and~\eqref{eq:sigma2} that 
$$
 \Sigma_k \ll    (d \omega)^k  N^{\rho+o(1)},
$$
which together with~\eqref{eq:diff}  yields
\begin{align*}
T_{\va, \bphi}( \vv; N)- T_{\va, \bphi}( \vu; N) & \ll N^{\rho+o(1)}   \sum_{k=1}^{\infty}  \frac{ (d \omega)^k}{k!}\\
& = N^{\rho+o(1)} \(\exp(d\omega) -1\).
\end{align*} 
Since $|\exp(\omega) -1|\ll \omega$ for  $\omega = O(1)$,   the desired result follows. 
\end{proof}

We remark that if $a_n = n^{o(1)}$ we can always apply Lemma~\ref{lem:con gen} with $\rho =1$, 
which we actually do in Lemma~\ref{lem:cont-gen}  below. 
On the other hand, we can use some $0<\rho<1$ for some special cases, see Lemma~\ref{lem:cont-linear}  below.
Furthermore, for applications of Lemma~\ref{lem:con gen} to Lemmas~\ref{lem:cont-gen} and~\ref{lem:cont-linear},  the value $\omega$ 
is quite small, in particular,    $\omega=o(1)$.  

For $\vu\in \R^d$ and $\bzeta=(\zeta_1, \ldots, \zeta_d)$ with  $\zeta_j>0$, $j=1, \ldots, d$, we define the  $d$-dimensional box centred at $\vu$ and
with the side lengths $2\bzeta$ by  
$$
\cR(\vu, \bzeta)=[u_1-\zeta_1, u_1+\zeta_1)\times \ldots \times [u_d-\zeta_d, u_d+\zeta_d).
$$

We have the following analogues of Wooley~\cite[Lemma 2.1]{Wool3}.

\begin{lemma} 
\label{lem:cont-gen} 
Suppose that  $\bphi \in \Z[T]^d$ and  $e_1, \ldots e_d$ are as~\eqref{eq:deg phi}.  
Let $0<\alpha<1$ and let $\varepsilon>0$ be  sufficiently small. 
Suppose that 
$|W_{\va, \bphi}(\vu; N)|\ge N^{\alpha}$ for  some $\vu\in \Tor$,  then for 
$$
0<\zeta_j\le N^{\alpha-e_j-1-\eps},  \qquad j =1, \ldots, d, 
$$ 
there is a set $\cR^*(\vu, \bzeta)\subseteq \cR(\vu, \bzeta)$ with  
$$
\lambda(\cR^*(\vu, \bzeta))\gg \lambda(\cR(\vu, \bzeta)),
$$
such that 
$$
|W_{\va, \bphi}( \vv; N)|\ge  N^{\alpha}/2
$$
holds for  any $\vv\in \cR^*(\vu, \bzeta)$  
provided  that $N$ is large enough. 
\end{lemma}

\begin{proof}  Let   $\cR^*(\vu, \bzeta)$ be the set of vectors $\vv\in \cR(\vu, \bzeta)$  which satsify  the ``non-negativity condition"~\eqref{eq:non-negative}. 
By Lemma~\ref{lem:con gen}, applied with $\rho = 1$ and $\omega = N^{\alpha-1-\eps}$,  for $\vv \in \cR^*(\vu, \bzeta)$ 
we have 
\begin{align*}
 \sum_{n=1}^{N} a_n  & \e\(hn/N + u_1 \varphi_1(n)+\ldots + u_d\varphi_d(n) \) \\
&  -   \sum_{n=1}^{N} a_n  \e\(h n/N + v_1 \varphi_1(n)+\ldots + v_d\varphi_d(n) \) \ll N^{\alpha - \varepsilon}.
\end{align*} 
The result  follows from the definition of $W_{\va, \bphi}(\vv; N)$ in Lemma~\ref{lem:control}.  
\end{proof}

\begin{lemma} 
\label{lem:cont-linear} 
Let  $\bphi \in \Z[T]^d$ be such that 
\begin{equation}
\label{eq:linear}
\min_{k<j\le d} \deg \varphi_j =1.
\end{equation}
Let $0 <\alpha<t\le 1$ 
and let $\varepsilon>0$ be  sufficiently small.  
Suppose that  $|W_{\va, \bphi}(\vu; N)|\ge N^{\alpha}$ for  some $\vu=(\vx, \vy)\in \Tor$ and  
\begin{equation} 
\label{eq:uniform-y}
\sup_{\vy\in \T_{d-k}} |T_{\va, \bphi}(\vx, \vy; M)|\le CM^t, \quad \forall 
\,  M\le N,
\end{equation}  
for some  constant $C$, then 
 for 
$$
0<\zeta_j\le N^{\alpha-e_j-t-\eps},  \qquad j =1, \ldots, d, 
$$ 
there is a set $\cR^*(\vu, \bzeta)\subseteq \cR(\vu, \bzeta)$ with  
$$
\lambda(\cR^*(\vu, \bzeta))\gg \lambda(\cR(\vu, \bzeta)),
$$
such that 
$$
|W_{\va, \bphi}( \vv; N)|\ge  N^{\alpha}/2
$$
holds for  any $\vv\in \cR^*(\vu, \bzeta)$  
provided  that $N$ is large enough. 
\end{lemma}

\begin{proof}   
From~\eqref{eq:linear}, without loss of generality, we assume that $\deg \varphi_d=1$ and hence 
$$
\varphi_d(n)=\varrho_1 n+\varrho_2
$$
for some real numbers $\varrho_1, \varrho_2$ with $\varrho_1\neq 0$. For any integer $h$  we write
\begin{equation}\label{eq:moving}
\begin{aligned}
 \e&\(hn/N+ u_1 \varphi_1(n)+\ldots + u_d\varphi_d(n) \)\\
 &=\e\( u_1 \varphi_1(n)+\ldots + (u_d\varrho_1+h/N) n+\varrho_2 u_d \).
\end{aligned}
\end{equation}
For any $1\le M\le N$, by~\eqref{eq:uniform-y}, with the vector of coefficients  
$$(u_1, u_2, \ldots,  u_{d-1}, u_d \varrho_1+h/N), $$
  we obtain 
$$
\left|\sum_{n=1}^{M} a_n   \e\( u_1 \varphi_1(n)+\ldots u_{d-1}\varphi_{d-1}(n)+ (u_d\varrho_1+h/N) n \) \right |\le CM^{t}.
$$
Combining this with~\eqref{eq:moving}  we derive  
$$
\left|\sum_{n=1}^{M} a_n   \e\(hn/N+ u_1 \varphi_1(n)+\ldots + u_d\varphi_d(n) \) \right| \le CM^t.
$$
By Lemma~\ref{lem:con gen}, applied with the coefficients $a_n\e(hn/N)$ instead of $a_n$, $\rho = t$ and $\omega = N^{\alpha-t-\eps}$, we have 
\begin{align*}
 \sum_{n=1}^{N} &a_n  \e\(hn/N+ u_1 \varphi_1(n)+\ldots + u_d\varphi_d(n) \) \\
&  \quad -   \sum_{n=1}^{N} a_n  \e\(hn/N + v_1 \varphi_1(n)+\ldots + v_d\varphi_d(n) \) \\
& \qquad\qquad \qquad\qquad\qquad \qquad \qquad\qquad \quad  \ll \omega N^{t} \ll N^{\alpha - \varepsilon}.
\end{align*} 
By the definition of $W_{\va, \bphi}(\vv; N)$ in Lemma~\ref{lem:control}  we obtain 
$$
|W_{\va, \bphi}( \vu; N)-W_{\va, \bphi}( \vv; N)|\ll N^{\alpha-\eps} \log N,
$$ 
the result now follows for all large enough $N$.
\end{proof}

We note that a similar concept of continuity of Weyl sums has also played a major role
in a different point of view on the distribution of Weyl sums~\cite{Brud,BD}.

\subsection{Distribution of large values of exponential sums}
\label{subsec:distribution}

We adapt the arguments of~\cite[Lemma~2.2]{Wool3} to our setting. 

First we show the following useful box counting result. We note that any better bound of the  exponent of $N$ 
 immediately yields an improvement of  our results.

Let $0<\alpha<1$ and let  $\eps$  be  sufficiently small. For each $j=1, \ldots, d$ let 
\begin{equation}
\label{eq:zetaj}
\zeta_j=1/ \rf{N^{e_j+1+\eps-\alpha}},
\end{equation}
where  $e_1, \ldots e_d$ are as~\eqref{eq:deg phi}.   

We divide $\Tor$ into 
$$
U = \(\prod_{j=1}^d \zeta_j\) ^{-1}
$$ 
boxes of the form
$$
[n_1\zeta_1, (n_1+1)\zeta_1)\times \ldots \times [n_d\zeta_d, (n_d+1)\zeta_d),
$$
where $n_j=1, \ldots, 1/\zeta_j$ for each $j=1, \ldots, d$. Let $\fR$ be the collection of these boxes, and 
\begin{equation}
\label{eq:set R}
\widetilde \fR=\{\cR \in \fR:~\exists\, \vu\in \cR \text{ with } |W_{\va, \bphi}( \vu; N)|\ge N^{\alpha}\}.
\end{equation}

\begin{lemma}
\label{lem:counting}
In the above notation,  we have 
$$
\# \widetilde \fR
 \le U N^{s(d)(1-2\alpha)+o(1)}.
$$
\end{lemma}

\begin{proof}
Let $\cR\in \fR$.  
By Lemma~\ref{lem:cont-gen}  if $|W_{\va, \bphi}( \vu; N)|\ge N^{\alpha}$ for some $\vu\in \cR$, 
then there is a set $\cR^*\subseteq \cR$  with   
$$
\lambda(\cR^*)\gg \lambda(\cR).
$$
such that  for any $\vv\in \cR^*$ we have $|W_{\va, \bphi}(\vv; N)|\ge N^{\alpha}/2$. 
Combining this with Corollary~\ref{cor:MVT-W}  we have   
$$
N^{ 2s(d) \alpha} \# \widetilde \fR\prod_{j=1}^{d}\zeta_j \ll \int_{\Tor} |W_{\va, \bphi}( \vu; N)|^{2s(d)}\, d\vu\le N^{s(d)+o(1)},
$$
which yields  the desired bound.
\end{proof}

Note that the above bound of $\# \widetilde \fR$ is nontrivial when $1/2<\alpha<1$.

\begin{cor}
\label{cor:Markov}
Let $0<\alpha<1$.  Then 
\begin{align*}
\lambda  (\{ \vx\in \T_k:~ \exists\, \vy \in \T_{d-k}  &\text{ with } |W_{\va, \bphi}(\vx, \vy; N)|\ge N^{\alpha}  \}   ) \\
 & \le N^{s(d)-2\alpha s(d)+\sigma_k(\bphi)+(d-k)(1-\alpha)+o(1)}.
\end{align*}
\end{cor}

\begin{proof}  We fix some sufficiently small   $\eps>0$ and define  the set 
$$
\fU = \bigcup_{\cR\in \widetilde \fR} \cR. 
$$
Observe that 
$$
\{ \vx\in \T_k:~\exists\, \vy \in \T_{d-k}  \text{ with }  |W_{\va, \bphi}(\vx, \vy; N)|\ge N^{\alpha}  \} 
 \subseteq \pi_{d, k} \( \fU \).
$$
Clearly we have 
$$
\lambda \left ( \pi_{d, k}   \(\fU \) \right ) \le \# \widetilde \fR \prod_{j=1}^{k} \zeta_j.
$$
By Lemma~\ref{lem:counting} and the choice of $\bzeta$ in~\eqref{eq:zetaj}, since $\eps$ is arbitrary, we now obtain the desired result. 
\end{proof}

Applying Lemma~\ref{lem:cont-linear}, in analogy with  Corollary~\ref{cor:Markov}, we obtain the  following.

\begin{cor}
\label{cor:box-linear}
Let  $\bphi \in \Z[T]^d$ be such that  
$$
\min_{k<j\le d} \deg \varphi_j =1.
$$
Let $\Omega\subseteq \T_k$ with $\lambda(\Omega)>0$ and let $0<\alpha<t\le 1$.  Suppose that  there exits a positive constant  $C$ such that  for all $\vx\in \Omega$  we have 
$$
\sup_{\vy\in \T_{d-k}} |T_{\va, \bphi}(\vx, \vy; N)| \le C N^{t},  \qquad \forall \, N\in \N.
$$ 
Let 
\begin{equation}
\label{eq:Bi}
\cB_{\Omega, N} = \left\{ \vx\in \Omega: ~ \exists\, \vy \in \T_{d-k} \text{ with } |W_{\va, \bphi}( \vx, \vy; N_{i})|\ge N^{\alpha}  \right\},
\end{equation}
then we obtain
$$
\lambda(\cB_{\Omega, N})\le  N^{s(d)-2\alpha s(d)+\sigma_k(\bphi)+(d-k)(t-\alpha) +o(1)}. 
$$
\end{cor}

\begin{proof} We fix some sufficiently small   $\eps>0$ and define
$$
  \zeta_j^*=1/ \rf{N^{j+t +\eps-\alpha}} \mand
 U^* =\( \prod_{j=1}^d \  \zeta_j^*\)^{-1},
$$ 
and we  divide $\Tor$ into $ U^*$ rectangles in a natural way. Let $\fR^{*}$ be the collection of these rectangles.

We also define an analogue of $\widetilde \fR$~\eqref{eq:set R} as 
$$
\widetilde {\fR}^{*}=\{\cR \in \fR:~\exists\, \vu\in \cR\text{ with }  \pi_{d, k} (\vu)\in \Omega \  \& \
|W_{d}( \vu; N)|\ge N^{\alpha}\}.
$$

Using Lemma~\ref{lem:cont-linear}  instead of Lemma~\ref{lem:cont-gen}  in the 
proof  Lemma~\ref{lem:counting} we obtain 
$$
\# \widetilde {\fR}^*\ll N^{s(d)-2s(d)\alpha} U^*.
$$
Applying the similar argument as in Corollary~\ref{cor:Markov} we obtain the result.
\end{proof}

\subsection{Orthogonal  projections of boxes}

We start with the following general result  which is perhaps well known.

\begin{lemma}
\label{lem:geometric}
Let $\cR\subseteq \R^d$ be a box  with the side lengths $h_1\ge \ldots \ge h_d$. Then  
for all $\cV \in \cG(d, k)$ we have  
$$
\lambda(\pi_{\cV}(\cR))\ll \prod_{i=1}^k h_i, 
$$
where the implied constant depends on $d$ and $k$ only. 
\end{lemma}

\begin{proof} The idea is to cover a box by balls,  and use that the size of the orthogonal  projections of any given ball  does not  depend on the \
choice of $\cV \in \cG(d, k)$. 

More precisely, without loss of  generality we can assume that 
$$
\cR=[0,h_1)\times \ldots \times  [0, h_d).
$$
Let
$$
\cR_k= [0,h_1)\times \ldots \times  [0, h_k)\times \{0\} \times \ldots \times \{0\}
$$
be a subset of $\cR$. Since for any $\vx\in \cR$ there exists $\vy\in \cR_k$ such that 
$$
 \|\vx-\vy\|\le \left (\sum_{j=k+1}^{d} h_j^{2}\right )^{1/2}\le d h_{k+1},
$$
we obtain 
\begin{equation}
\label{eq:covering1}
\cR \subseteq \cR_k + \cB\({\bf 0}, d h_{k+1}\), 
\end{equation}
where $\cB({\bf 0}, d h_{k+1})$ is the ball of $\R^d$  centred at ${\bf 0}$ and of radius $d h_{k+1}$ and for $\cA, \cB \subseteq \R^d$, as usual, we define: 
$$
\cA+\cB=\{\va+\vb:~\va\in \cA,\ \vb \in \cB\}.
$$

Now we intend to cover $\cR_k$ by a family of balls of $\R^d$ such that each of these balls has the radius roughly $ h_{k+1}$.

For each $1\le j\le k$ we have 
$$
[0, h_j)\subseteq \bigcup_{q=0}^{Q_j} \cI_{j ,q}, 
$$
where $ \cI_{j ,q} =
[qh_{k+1}, (q+1)h_{k+1})
$ and 
\begin{equation}
\label{eq:number}
Q_j=\rf{h_j/h_{k+1}}.
\end{equation}
Then 
$$
\cR_k \subseteq \bigcup_{0\le q_1\le Q_1, \ldots, 0\le q_k \le Q_k} \cI_{1, q_1} \times \ldots \times \cI_{k, q_k}\times \underbrace{\{0\}\times \ldots \times \{0\}}_{d-k~\text{times}}.
$$
Observe that for each choice on integers $q_1, \ldots, q_k$ with 
$$
0\le q_1\le Q_1, \ldots, 0\le q_k \le Q_k,
$$ 
there exists a ball $\cB_{q_1, \ldots, q_k}$ of $\R^d$ of radius $d h_{k+1}$ such that 
$$
\cI_{1, q_1} \times \ldots \times \cI_{k, q_k}\times \{0\}\times \ldots \times \{0\} \subseteq \cB_{q_1, \ldots, q_k}.
$$
Denote the collection of these balls by
$$
\fB=\{\cB_{q_1, \ldots, q_k}:~0\le q_1\le Q_1, \ldots, 0\le q_k \le Q_k\}.
$$
It follows that 
\begin{equation}
\label{eq:covering2}
\cR_k \subseteq \bigcup_{\cB\in \fB} \cB.
\end{equation}
Since the radius of each  ball $\cB\in \fB$  is $d h_{k+1}$, we have  
$$
\cB+\cB({\bf 0}, d h_{k+1})\subseteq 2\cB, 
$$
where $2 \cB(\vx, r)=\cB(\vx, 2r)$. 
Together with~\eqref{eq:covering1} and~\eqref{eq:covering2} we obtain 
$$
\cR\subseteq \cR_k + \cB\({\bf 0}, d h_{k+1}\) \subseteq  \bigcup_{\cB\in \fB} 2 \cB.
$$

It follows that for any $\cV\in \cG(d, k)$ we have 
$$
\pi_{\cV}(\cR)\subseteq \bigcup_{\cB\in \fB}  \pi_{\cV} (2 \cB).
$$
Since for each ball $\cB\in \fB$ the projection $\pi_{\cV} (2 \cB)$ is a ball of the $k$-dimensional  subspace $\cV$ with radius $2d r_{k+1}$, one has
$$
\lambda \left (\pi_{\cV} (2 \cB)  \right ) \ll h_{k+1}^{k}.
$$
Combining this with~\eqref{eq:number}, we obtain  
$$
\lambda (\pi_{\cV}(\cR))\ll h_{k+1}^{k}\,\prod_{j=1}^k Q_j \ll \prod_{i=1}^{k}h_i, 
$$  
which gives the result. 
\end{proof}

\subsection{Orthogonal projections and large values of exponential sums}  

We now   provide a basic tool for the proof of Theorem~\ref{thm:B}.
Applying Lemma~\ref{lem:counting} and Lemma~\ref{lem:geometric}  we obtain the following 
analogue of Corollary~\ref{cor:Markov}.

\begin{cor}
\label{cor:general-p}
Let $0<\alpha<1$.  
For  any $\cV\in \cG(d, k)$ we have 
\begin{align*}
\lambda  (\{ \vx\in \cV:~ \exists\, \vu \in \Tor  \text{ with }   \pi_\cV(\vu) &= \vx  \  \& \  W_{\va, \bphi}( \vu; N)|\ge N^{\alpha}  \}   ) \\
 & \le N^{s(d)-2\alpha s(d)+\widetilde{\sigma}_k(\bphi)+(d-k)(1-\alpha) +o(1)},
\end{align*}
as $N \to \infty$, where $\widetilde{\sigma}_k(\bphi)$ is given by~\eqref{eq:wsigma}.
\end{cor}

\begin{proof} We fix some sufficiently small   $\eps>0$. 
 We use the same notation as in Section~\ref{subsec:distribution}, including  the choice of $\zeta_j$,  $j=1, \ldots, d$
in~\eqref{eq:zetaj}.
For $\cR\in \mathfrak{R}$ with the side lengths $\zeta_1, \ldots, \zeta_d$ we denote them as 
$$
\widetilde{\zeta}_1\ge \ldots \ge \widetilde{\zeta}_d.
$$ 
 For $j=1, \ldots, d$  by~\eqref{eq:degree} we obtain  
\begin{equation}
\label{eq:order}
\widetilde{ \zeta_j}=1/ \rf{N^{r_j+1+\eps-\alpha}}.
\end{equation}
We also define the set 
$$
\fU = \bigcup_{\cR\in \widetilde \fR} \cR. 
$$
Observe that 
$$
\{ \vx\in \cV:~\exists\, \vu \in \Tor \text{ with } \pi_\cV(\vu) = \vx   \ \& \   |W_{\va, \bphi}( \vu; N)|\ge N^{\alpha}  \} \subseteq \pi_{\cV}\(\fU \).
$$ 
Combining this 
with Lemma~\ref{lem:counting}, Lemma~\ref{lem:geometric} and~\eqref{eq:order} we obtain
\begin{align*}
\lambda\( \pi_{\cV}\(\fU \) \) \le \# \widetilde \fR \prod_{i=1}^{k} \widetilde{\zeta}_i &\le N^{s(d)-2\alpha s(d)+o(1)} \prod_{j=k+1}^{d}\widetilde{\zeta}_j ^{-1}\\
&\qquad \le N^{s(d)-2\alpha s(d)+o(1)} \prod_{j=k+1}^{d} N^{r_j+1+\eps-\alpha}.
\end{align*}
By the definition  of $\widetilde{\sigma}_k(\bphi)$ and since $\eps$ is arbitrary, we obtain the desired bound.
\end{proof}

\section{Proofs of exponential sum bounds}

\subsection{Proof of Theorem~\ref{thm:general}}

We  fix some  $\alpha > 1/2 $  and set
$$
N_i =    2^i, \qquad i =1, 2, \ldots.
$$

We now consider the set 
$$
\cB_{i} = \left\{ \vx\in \T_k:~ \exists\, \vy \in \T_{d-k} \text{ with } |W_{\va, \bphi}( \vx, \vy; N_{i})|\ge N_i^{\alpha}  \right\}. 
$$
By Corollary~\ref{cor:Markov}  we have  
\begin{equation}
\label{eq:lambdaBi}
\lambda\(\cB_i\)   \le N_i^{s(d)-2\alpha s(d)+\sigma_k(\bphi)+(d-k)(1-\alpha) +o(1)}. 
\end{equation} 
We ask that  the parameters satisfy the following condition 
\begin{equation}
\label{eq:condition-general}
s(d)-2\alpha s(d)+\sigma_k(\bphi)+(d-k)(1-\alpha)<0.
\end{equation}
Combining~\eqref{eq:condition-general}  with  the {\it Borel--Cantelli lemma\/}, we obtain that  
$$
\lambda \left(\bigcap_{q=1}^{\infty}\bigcup_{i=q}^{\infty} \cB_i \right)=0.
$$
Since
\begin{align*}
\{ \vx\in \T_k:~\exists\, \vy \in \T_{d-k} & \text{ with } |W_{\va, \bphi}(\vx, \vy; N_i)|\ge N_i^{\alpha}  \\
&\text{ for infinite many } i\in \N \} \subseteq \bigcap_{q=1}^{\infty}\bigcup_{i=q}^{\infty} \cB_i,
\end{align*}
we conclude that for almost all $\vx\in \T_k$ there exists $i_{\vx}$ such that for any $i\ge i_{\vx}$  one has  
\begin{equation}
\label{eq:y}
\sup_{\vy\in \T_{d-k}}|W_{\va, \bphi}(\vx, \vy; N_i)|\le N_i^{\alpha}.
\end{equation}

We fix this $\vx$ in the following arguments. For any $N\ge N_{i_{\vx}}$ there exists $i$ such that 
$$
N_{i-1}\le N< N_{i}.
$$
By Lemma~\ref{lem:control} and~\eqref{eq:y} we have 
\begin{align*}
\sup_{\vy\in \T_{d-k}}|T_{\va, \bphi}(\vx, \vy; N)|  \ll \sup_{\vy\in \T_{d-k}} |W_{\va, \bphi}(\vx, \vy; N_i)| \ll N^{\alpha}.
\end{align*}

Note that the condition~\eqref{eq:condition-general} 
can be written as
$$
\alpha>\frac{s(d)+\sigma_k(\bphi)+d-k}{2s(d)+d-k},
$$
which gives the desired bound.

\subsection{Proof of Theorem~\ref{thm:A}} 

Recall  that  
$$
\Gamma_{YL}(\bphi, k)=\frac{1}{2}+\frac{\sigma_k(\bphi)}{2s(d)}.
$$
Applying a  similar chain of arguments  as the proof of Theorem~\ref{thm:general},  from Corollary~\ref{cor:box-linear}  we derive the following ``self-improving'' property of Weyl sums. 

\begin{lemma}
\label{lem:self}
Let  $\bphi \in \Z[T]^d$ be such that 
$$
\min_{k<j\le d} \deg \varphi_j =1.
$$
Let $\Omega\subseteq \T_k$ with $\lambda(\Omega)>0$ and let $0<t\le 1$.  Suppose that  there exits a positive constant  $C$ such that  for all $\vx\in \Omega$  we have 
$$
\sup_{\vy\in \T_{d-k}} |T_{\va, \bphi}(\vx, \vy; N)| \le C N^{t},  \qquad \forall \, N\in \N.
$$
Then for almost all $\vx\in \Omega$ and for any $\varepsilon>0$ there exists a positive constant $C(\vx, \varepsilon)$ such that 
$$
\sup_{\vy\in \T_{d-k}} |T_{\va, \bphi}(\vx, \vy; N)| \le  C(\vx, \varepsilon)N^{f(t)+\varepsilon},  \qquad \forall \, N\in \N,
$$
where 
$$
f(t)=\frac{s(d)+\sigma_k(\bphi)+(d-k)t}{2s(d)+d-k}.
$$
\end{lemma} 

\begin{proof}
We  fix some  $0<\alpha <t$  and set
$$
N_i =    2^i, \qquad i =1, 2, \ldots.
$$
For each $i$ denote  
$$
\cB_{\Omega, N_i} = \left\{ \vx\in \Omega:~ \exists\, \vy \in \T_{d-k} \text{ with } |W_{\va, \bphi}( \vx, \vy; N_{i})|\ge N_i^{\alpha}  \right\}.
$$
Corollary~\ref{cor:box-linear} gives 
$$
\lambda(\cB_{\Omega, N_i})\le  N_i^{s(d)-2\alpha s(d)+\sigma_k(\bphi)+(d-k)(t-\alpha) +o(1)}. 
$$   
Similarly to the proof of Theorem~\ref{thm:general}, we ask the parameters satisfy the condition 
$$
s(d)-2\alpha s(d)+\sigma_k(\bphi)+(d-k)(t-\alpha)<0,
$$
which is 
$$
\alpha>\frac{s(d)+\sigma_k(\bphi)+(d-k)t}{2s(d)+d-k}.
$$
Thus we finishes the proof.
\end{proof}

We remark that in Lemma~\ref{lem:self} if $t>\Gamma_{YL}( \bphi, k)$ then $f(t)<t$, this is reason 
why we call it  a ``self-improving'' type result.

We now immediately  derive from Lemma~\ref{lem:self} the following ``self-improving'' property  underlying our bounds. 
Compared to Lemma~\ref{lem:self} it allows us to have  some level of non-uniformity in 
our assumption.   

\begin{cor}\label{cor:self}
Let  $\bphi \in \Z[T]^d$ be such that 
$$
\min_{k<j\le d} \deg \varphi_j =1.
$$
Let $\Omega\subseteq \T_k$ with $\lambda(\Omega)>0$ and let $0<t\le 1$. Suppose that  for almost all $\vx\in \Omega$ there exits a positive constant  $C(\vx)$ such that 
$$
\sup_{\vy\in \T_{d-k}} |T_{\va, \bphi}(\vx, \vy; N)| \le C(\vx) N^{t},  \qquad \forall \, N\in \N.
$$  
Then for almost all $\vx\in \Omega$ and for any $\varepsilon>0$ there exists a positive constant $C(\vx, \varepsilon)$ such that 
$$
\sup_{\vy\in \T_{d-k}} |T_{\va, \bphi}(\vx, \vy; N)| \le  C(\vx, \varepsilon)N^{f(t)+\varepsilon},  \qquad \forall \, N\in \N,
$$
where 
$$
f(t)=\frac{s(d)+\sigma_k(\bphi)+(d-k)t}{2s(d)+d-k}.
$$
\end{cor} 

\begin{proof}
We take a decomposition $\Omega=\bigcup_{q=0}^{\infty} \Omega_{q}$ such that $\lambda(\Omega_0)=0$ and for each $q\ge 1$  
the sums are uniformly bounded by $qN^{t}$,   that is, for any $\vx\in \Omega_q$ we have 
$$
\sup_{\vy\in \T_{d-k}} |T_{\va, \bphi}(\vx, \vy; N)| \le q N^{t},  \qquad \forall \, N\in \N.
$$
Applying Lemma~\ref{lem:self} for each $\Omega_q$, $q\ge 1$,  we obtain the desired result.
\end{proof}

Now we turn to the proof of Theorem~\ref{thm:A}. Denote 
$$
f(x)=\frac{s(d)+\sigma_k(\bphi)+(d-k)x}{2s(d)+d-k}.
$$
Firstly Theorem~\ref{thm:general} claims that for almost all $\vx\in \T_k $  and for any $\eps>0$ there exists a constant $C(\vx, \eps)$ such that 
$$
\sup_{\vy\in \T_{d-k}} |T_{\va, \bphi}(\vx, \vy; N)| \le C(\vx) N^{f(1)+\eps},  \qquad \forall \, N\in \N.
$$
Applying Corollary~\ref{cor:self} repeatedly, we obtain the following sequence 
$$
1\rightarrow f(1)\rightarrow f(f(1))\rightarrow f(f(f(1))) \rightarrow\ldots\, .  
$$
Since the function $t \mapsto f(t)$ is strictly monotonically decreasing for
$$
t > \frac{1}{2}+\frac{\sigma_k(\bphi)}{2s(d)} = \Gamma_{YL}(\bphi,k), 
$$
and  
$$
f\(  \Gamma_{YL}(\bphi,k) \) =  \Gamma_{YL}(\bphi,k),
$$
by an arbitrary choice small enough $\eps>0$ at each steps, we finish the proof.

\subsection{Proof of Theorem~\ref{thm:weyl short}} 
For $M\in \Z$, recall that  Weyl sums over short intervals  are defined as follows 
$$
S_d(\vu; M,N) = \sum_{n=M+1}^{M+N} \e(u_1n+\ldots+u_d n^d). 
$$
 We write 
$$
S_d(\vu; M,N)=\sum_{n=1}^N \e(u_1(n+M)+\ldots+u_d(n+M)^d),
$$
and observe that in the  polynomial identity  
\begin{equation}
\label{eq:shift}
\begin{split}
u_1(T+M)+\ldots & +u_d (T+M)^d \\
& = v_0+v_1T+\ldots+v_{d-1}T^{d-1}+u_d T^d \in \R[T], 
\end{split}
\end{equation}
where for $j=0, 1, \ldots, d-1$, each $v_j$, depends only on $u_1, \ldots, u_{d}$ and $M$. 

Hence Theorem~\ref{thm:A}, applied $k=1$, $\varphi_1(T)=T^d$  
and thus with $\sigma_{1} (\bphi) = d(d-1)/2$,  
yields the desired estimate on $S_d(\vu; M,N)$.

\subsection{Proof of Theorem~\ref{thm:B}}
As we have claimed, Theorem~\ref{thm:B} now follows by 
applying  Corollary~\ref{cor:general-p}  instead of Corollary~\ref{cor:Markov}  and using similar arguments as in the proofs of 
Theorem~\ref{thm:general}. We omit these very similar arguments here.

\section{Proof of discrepancy bounds}

\subsection{Preliminaries}

We start with recalling  the classical {\it Erd\H{o}s--Tur\'{a}n inequality\/} (see, for instance,~\cite[Theorem~1.21]{DrTi}).

\begin{lemma}
\label{lem:ET}
Let $\xi_n$, $n\in \N$,  be a sequence in $[0,1)$. Then for the discrepancy $D_N$ given by~\eqref{eq:Discr} and any $G\in \N$, we have  
$$
D_N \le 3 \left( \frac{N}{G+1} + \sum_{g=1}^{G}\frac{1}{g} \left| \sum_{n=1}^{N} \e(g \xi_n) \right | \right).
$$
\end{lemma}

We also use the following trivial property of the  Lebesgue measure, see~\cite[Section~3]{Wool3}
for a short proof. 

\begin{lemma}
\label{lem:invairant}
Let $\cA\subseteq \Tor$ and $g\in \N$, then 
$$
\lambda  (\{\vx\in \Tor:~(g\vx \pmod 1) 
 \in \cA\} )=\lambda(\cA).
$$
\end{lemma}

\subsection{Proof of Theorem~\ref{thm:general-D}} As in Section~\ref{subsec:distribution}, if $\va=\ve=(1)_{n=1}^{\infty}$,  we just write  
$$
W_{\bphi}( \vx, \vy; N)=W_{\ve,\bphi}( \vx, \vy; N).
$$

Let $N_i=2^i$, $i\in \N$ and  let 
$G_i=\fl{N_i^{\eta}}$ 
for some $\eta>0$ to  be chosen later.  For each $g=1, \ldots, G_i$ let 
$$
\cB_{i, g}= \left\{ \vx\in \T_k:~ \exists\, \vy \in \T_{d-k} \text{ with } |W_{\bphi}( g\vx, g\vy; N_{i})|\ge N_i^{\alpha}  \right\},
$$ 
and 
$$
\widetilde{\cB_i}=\bigcup_{g=1}^{G_i} \cB_{i, g}.
$$
Observe that 
$$
\cB_{i, g}=\{\vx\in \T_k:~(g\vx \pmod 1) \in \cB_i\}, 
$$ 
where the notation $\cB_i$ is  given by~\eqref{eq:Bi} in  the case $\va=\ve$. 
 By Lemma~\ref{lem:invairant} and the inequality~\eqref{eq:lambdaBi} we conclude that 
$$
\lambda(\widetilde{\cB_i})\le G_i N_i^{s(d)-2\alpha s(d)+\sigma_k(\bphi)+(d-k)(1-\alpha)+o(1)}. 
$$
We ask that  the fixed $\alpha$ and $\eta$ satisfy the following condition 
\begin{equation}
\label{eq:con1}
\eta+s(d)-2\alpha s(d)+\sigma_k(\bphi)+(d-k)(1-\alpha)<0.
\end{equation}
Combining this with  the Borel--Cantelli lemma, and choosing a small enough $\eps$, we obtain that 
$$
\lambda \left(\bigcap_{q=1}^{\infty}\bigcup_{i=q}^{\infty} \widetilde{\cB_i}\right)=0.
$$
It follows that  for almost all $\vx\in \T_k$ there exists $i_{\vx}$ such that for any $i\ge i_{\vx}$  and 
any $g=1, \ldots,  G_i$, one has 
$$
\sup_{\vy\in \T_{d-k}}|W_{ \bphi}( g\vx, g\vy; N_i)|\le N_i^{\alpha}.
$$
Combining with Lemma~\ref{lem:control} we obtain that  for any $N\ge N_{i_{\vx}}$ there exists $i\in \N$ such that 
$$
N_{i-1}\le N < N_i,
$$
and  for any $g=1, \ldots,  G_i$, one has 
$$
\sup_{\vy\in \T_{d-k}}|T_{ \bphi}( g\vx, g\vy; N)|\ll \sup_{\vy\in \T_{d-k}}|W_{\bphi}( g\vx, g\vy; N_i)| \ll N^{\alpha}.
$$
Applying  Lemma~\ref{lem:ET}  for  $N$ and $G=G_i$ we conclude that  
$$
\sup_{\vy\in \T_{d-k}} D_{\bphi}(\vx, \vy; N)\ll N /G_i + N^{\alpha} \log G_i
\ll N^{1-\eta} + N^{\alpha} \log  N. 
$$
Let  $\eta=1-\alpha$. The condition~\eqref{eq:con1} can be written as
$$
\alpha>\frac{\sigma_k(\bphi)+s(d)+d-k+1}{2s(d)+d-k+1},
$$
which finishes  the proof.

\subsection{Proof of Theorem~\ref{thm:discrep short}}
Recall that  $D_d(\vu;M, N)$  is  the discrepancy  of the sequence  of fractional parts
$$
\{u_1n+\ldots+u_d n^d\}, \qquad  n=M+1, \ldots, M+N.
$$
Clearly this sequence is same as  
$$
\{u_1(n+M)+\ldots+u_d (n+M)^d\}, \qquad  n=1, \ldots, N, 
$$
and thus as  before, see~\eqref{eq:shift}, we see that this sequence is the same as   
$$
\{v_0+v_1n+\ldots+v_{d-1}n^{d-1}+u_d n^d\}, \qquad  n=1, \ldots, N, 
$$
where for $j=0, 1, \ldots, d-1$, each $v_j$, depends only on $u_1, \ldots, u_{d}$ and $M$. 
Furthermore let $\vu^*=(v_1, \ldots, v_{d-1}, u_d)$   and 
$$
D_d(\vu; N)=D_d(\vu; 0, N).
$$
Then we have 
$$
D_d(\vu^*; N)\ll D_d(\vu; M, N)\ll D_d(\vu^*; N),
$$
where the implied constant is  absolute. 
This can be showing by combining the above arguments and the following ``translation invariance" of the discrepancy. More precisely, let $\xi$ be a constant and $\xi_n$ be a sequence of real number.  Let  $D_{\xi,N}$ be the discrepancy of the fractional parts
$$
\{\xi+\xi_n\}, \qquad  n=1, \ldots, N.
$$ 
Thus $D_N = D_{0,N}$.
From the definition of discrepancy~\eqref{eq:Discr} we derive 
$$
D_N\ll D_{\xi,N}\ll D_N.
$$

Easy calculations, show that Theorem~\ref{thm:general-D} with $k=1$, $\varphi_1(T)=T^d$  
and thus with $\sigma_{1} (\bphi) = d(d-1)/2$, implies the result.

\section{Comments}
\label{sec:comments}

\subsection{Discrepancy of polynomials}  
First of all we give a proof for~\eqref{eq:Du}, that is, that  for almost all $\vu\in \Tor$ one has 
$$
D_d(\vu; N)\le N^{1/2} (\log N)^{3/2 + o(1)}, \qquad N \to  \infty.
$$

This is based on~\cite[Theorem~5.13]{Harm} (see Proposition~\ref{prop:Harm} below) and  the following 
general statement which is perhaps well-known but the authors have not been able to find it in the literature. 

\begin{prop} 
\label{pro:pro}
Let $\cU \subseteq \Tor$, be  a set of positive Lebesgue measure $\lambda(\cU) > 0$. Then
 there is a vector $\vv_0 \in \Tor$ and a set  real numbers $\cW  \subseteq [0, \sqrt{d}]$  of positive Lebesgue measure
  $\lambda(\cW) > 0$,  such that for every $w \in \cW$ we have 
$ w \vv_0 \in \cU$. 
\end{prop}

\begin{proof}
Let $\chi_{\cU}$ be the characteristic function of $\cU$, then clearly we have 
\begin{equation}
\label{eq:measureU}
\lambda(\cU) = \int_{\Tor} \chi_{\cU}(\vu)\, d\vu.
\end{equation}
Applying the polar coordinates~\cite[Theorem~3.12]{EG} to the function $\chi_{\cU}$,  we obtain 
\begin{equation}
\label{eq:uu}
\int_{\Tor} \chi_{\cU}(\vu)\, d\vu=\int_{0}^{\sqrt{d}} \left(\int_{\{\vu: \|\vu\|=r\}} \chi_{\cU}(\vu) d \mathcal{H}^{d-1}(\vu)\right)dr, 
\end{equation}
where $\mathcal{H}^{d-1}$ is the $(d-1)$-dimensional Hausdorff measure which is given by~\cite[Chapter~2]{EG}.  
By taking $\vu=r\vv$ for some $\vv \in \mathbb{S}^{d-1}$ in the second term of~\eqref{eq:uu}  we obtain 
$$
\int_{\{\vu: \|\vu\|=r\}} \chi_{\cU}(\vu) d \mathcal{H}^{d-1}(\vu)=\int_{\mathbb{S}^{d-1}} \chi_{\cU}(r\vv) r^{d-1}d \mathcal{H}^{d-1}(\vv), 
$$
where $\mathbb{S}^{d-1}\subseteq \R^d$ is the unit sphere centred at the origin.  
Combining this with~\eqref{eq:measureU} and~\eqref{eq:uu}  and applying  {\it Fubini's theorem\/},  we arrive to  
$$
\lambda(\cU) = \int_{\mathbb{S}^{d-1}} \left(\int_{0}^{\sqrt{d}} \chi_{\cU}(r\vv) r^{d-1} dr \right) d \mathcal{H}^{d-1}(\vv).
$$
Since $\lambda(\cU)>0$, we conclude that there exist a vector 
$$
\vv=(v_1, \ldots, v_d)\in \mathbb{S}^{d-1}
$$  
and a set   of $r \in [0,\sqrt{d}]$ of positive Lebesgue measure such that $r\vv\in \cU$, which gives the desired result. 
\end{proof}

We formulate~\cite[Theorem~5.13]{Harm}, see also~\cite{Baker}, in the following form.

\begin{prop} 
\label{prop:Harm}
Let $\cA = \(a_n\)_{n=1}^\infty$ be a sequence increasing sequence of real numbers such that $a_{n+1}-a_n\ge \delta>0$ and let $\eps>0$. Then for almost all $w\in \R$ we have 
$$
D\(w\cA; N\)\ll   N^{1/2} (\log N)^{3/2 + \varepsilon},
$$
where $D\(w\cA; N\)$ means the discrepancy of the sequence $w a_n \pmod 1 $, $n =1, \ldots, N$. 
\end{prop}

Let us now fix some   $\varepsilon >0$ and denote by 
$\cU \subseteq \Tor$  the set of $\vu \in  \Tor$, for which 
\begin{equation} 
\label{eq:Du eps}
D_d(\vu; N)\ge  N^{1/2} (\log N)^{3/2 + \varepsilon},
\end{equation}
for infinitely many $N \in \N$.  Assume that $\lambda (\cU)> 0$. By  Proposition~\ref{pro:pro}  there exists a vector 
$$
\vv=(v_1, \ldots, v_d)\in \Tor
$$ 
and a set   of $w \in [0, \sqrt{d}]$ of positive Lebesgue measure such that we have~\eqref{eq:Du eps} with $\vu = w \vv $  and for infinitely many $N \in \N$.  On the other hand applying Proposition~\ref{prop:Harm} to the sequence $v_1n+\ldots +v_d n^{d}$, $n\in \N$ and parameter $\varepsilon/2$, we obtain that for almost all $w\in \R$ one has 
$$
D_d(w\vv; N)\ll N^{1/2} (\log N)^{3/2 + \varepsilon/2}.
$$  
This now gives the contradiction and therefore together with the arbitrary choice of $\varepsilon$ the estimate~\eqref{eq:Du}  holds. 

At the moment we are not able to rule out that for almost all $\vu\in \Tor$ one has 
$$
D_d(\vu; N)\le N^{o(1)}, \qquad N \to  \infty,
$$
for any $d\ge 3$, which we believe to be false. In fact, as we have mentioned,  we believe that~\eqref{eq:Du}
is tight except the logarithm factor, and this is true for the case $d=2$ which follows from a result of 
Fedotov and Klopp~\cite[Theorem~0.1]{FK}  and the  Koksma inequality~\cite[Theorem~5.4]{Harm}, 
while the conjecture is still open when $d\ge 3$.

For the monomial sequence $xn^d$, $n\in \N$ we denote by $D_d(x; N)$ the corresponding discrepancy of its fractional parts. We note that in the case $d=1$ the celebrated result of  Khintchine, see~\cite[Theorem~1.72]{DrTi}, implies that  for almost all $x\in [0,1)$ one has 
$$
D_1(x; N)\le N^{o(1)}, \qquad N \to  \infty.
$$

Finally, we remark that for $d\ge 2$  Aistleitner and Larcher~\cite[Corollary~1]{AL} have recently shown that  for almost all $x\in [0,1)$  one has 
$$
D_d(x; N)\ge N^{1/2-\eps}
$$
 for any $\eps>0$ and for infinitely many $N\in \N$. Many other metrical results on the discrepancy of polynomials and other sequences
 can be found in~\cite{AL,Bil,Harm}. 
 
\subsection{The structure of the exceptional sets}  
For $1\le k<d$ and $0<\alpha<1$ denote 
\begin{align*}
\cE_{\bphi, k, \alpha} =\{\vx\in \T_k:~\sup_{\vy\in \T_{d-k}} |T_{\bphi}(\vx, \vy;& N)|\ge N^{\alpha} \\
&\text{ for infinitely many } N\in \N\}.
\end{align*}
Theorem~\ref{thm:general} claims that for any 
$$
\alpha >\frac{1}{2}+\frac{2\sigma_k(\bphi)+d-k}{2d^2+4d-2k},
$$
the set 
$
\cE_{\bphi, k, \alpha}
$
is of zero $k$-dimensional Lebesgue measure.   
It is natural to ask what we can  say more for these  sets with zero Lebesgue measure. 

Motivated from the works~\cite{ChSh, ChSh2} we ask  the size of $\cE_{\bphi, k, \alpha}$ in the sense of Baire categories and Hausdorff dimension. In the following suppose that $\bphi$ is the classical choice as in~\eqref{eq:classical}. 
The argument in~\cite{ChSh}  implies that for any $1\le k<d$ and any $0<\alpha<1$ the set
$$
\T_k\setminus \cE_{\bphi, k, \alpha}
$$
is of first Baire category in $\T_k$.  For the Hausdorff dimension,~\cite[Corollary~1.3]{ChSh2} implies that 
$$
\dim \cE_{\bphi, k, \alpha} \rightarrow 0 \quad \text{as}  \quad \alpha \rightarrow 1,
$$
where $\dim$ means the Hausdorff dimension. We omit these details here.  

\subsection{Further possible extensions}

We have formulated  Theorems~\ref{thm:general} and~\ref{thm:A}  in terms of the unit torus $\T_d$ only. In fact these result 
may shed some light  for subsets of $\T_d$ also. For instance, 
 Theorem~\ref{thm:general} implies the following statement. 
 
 Let $\cA \subseteq \T_d$ such that 
$
\lambda(\pi_{d, k}(\cA))>0,
$
where the notation $\pi_{d, k}$ is given by~\eqref{eq:pidk} and the symbol $\lambda$  represents the $k$-dimensional Lebesgue measure. Note that the set $\cA$ itself may be of vanishing  $d$-dimensional Lebesgue measure.
Suppose that $\bphi \in \Z[T]^d$ is such that  the Wronskian  $W(T;\bphi)$ 
does not vanish identically, then for almost all $\vx\in \pi_{d, k}(\cA)$ one has
\begin{equation}
\label{eq:fiber}
\sup_{\vy\in \T_{d-k}} |T_{\va, \bphi}(\vx, \vy; N)| \le N^{\Gamma(\bphi,k) +o(1)},  \qquad N\to \infty,
\end{equation}
where $\Gamma(\bphi,k)$ comes from  Theorem~\ref{thm:general}.

Indeed, let $\cG\subseteq \T_d$ be the collection of vectors $\vx\in \T_k$ which satisfy~\eqref{eq:fiber}, then Theorem~\ref{thm:general} implies that the set $\cG$ has full measure (that is, $\lambda(\cG)=1$), and therefore  
$$
\lambda(\pi_{d, k}(\cA) \cap \cG)=\lambda(\pi_{d, k}(\cA)).
$$
Thus the bound~\eqref{eq:fiber} holds for almost all $\vx\in \pi_{d, k}(\cA)$.

Now, suppose that $\lambda(\pi_{d, k}(\cA))=0$ for some subset $\cA\subseteq \T_d$. 
It is interesting to investigate whether one can obtain an appropriate analogue  of Theorem~\ref{thm:general} in this case. 
 In the following we formulate a general framework for the possible extension of Theorem~\ref{thm:general}.

Let $\cA\subseteq \T_d$ and $\mu$ be a ``nice" probability measure on $\cA$, for example a {\it Borel measure\/}.  Suppose that the measure $\mu$ admits  some kind
 of the mean value theorem, that is, there exist  positive constants $s<t$ such that  ($a_n=n^{o(1)}$)
\begin{equation}
 \label{eq:mu-VMT}
 \int_{\cA} |T_{a, \bphi}(\vu; N)|^t d\mu(\vu)\le N^{s+o(1)}.
\end{equation} 
Furthermore,  
assume that
$\mu$ has some regular properties, for instance, there exist positive constants $\beta_1<\beta_2$ such that for any ball $\cB(\vu, r)$ 
centred at $\vu\in \cA$ and of positive radius   $r<1$ one has  
\begin{equation}
\label{eq:mu-regular}
r^{\beta_2}\ll \mu(\cB(\vu, r))\ll r^{\beta_1}
\end{equation}
with some absolute implied constants. Note that the condition~\eqref{eq:mu-regular} gives the upper bound  and lower bound on the measure of any given  high dimensional rectangle. 

We remark that our methods work for any subset $\cA\subseteq \T_d$ and any measure $\mu$ on $\cA$ which has the above properties~\eqref{eq:mu-VMT} and~\eqref{eq:mu-regular}. More precisely, let $\mu_{d, k}$ be the projection measure of $\pi_{d, k}$, that is, 
$$
\mu_{d, k}(\cF)=\mu\{\vu\in \T_d: ~\pi_{d, k}(\vu)\in \cF\}, \quad \cF\subseteq \T_k.
$$
Suppose that $\bphi \in \Z[T]^d$ is such that  the Wronskian  $W(T;\bphi)$ 
does not vanish identically, then for $\mu_{d, k}$-almost all $\vx\in \pi_{d, k}(\cA)$ one has
$$
\sup_{\vy\in \T_{d-k}} |T_{\va, \bphi}(\vx, \vy; N)| \le N^{\Gamma(\cA, \mu, \bphi) +o(1)},  \qquad N\to \infty,
$$
where $\Gamma(\cA, \mu, \bphi)$ is a positive constant which can be explicitly evaluated in terms 
of the parameters in~\eqref{eq:mu-VMT} and~\eqref{eq:mu-regular}.  
We  expect that 
$$
\Gamma(\cA, \mu, \bphi)<1
$$ 
holds in many natural situations. 

On the other hand, it is not clear how to extend Theorem~\ref{thm:general-D} (the result of the discrepancy)  to    subsets $\cA\subseteq \T_d$ with some measure $\mu$ on $\cA$ since in general  $\mu$ does not have the invariant property as in Lemma~\ref{lem:invairant}.   
For example, the famous {\it $\times 2 \times 3$-Conjecture\/} of Furstenberg, which still remains open (see~\cite{Schm,Wu} and references therein): 
using our notation as in  Lemma~\ref{lem:invairant}, 
let $\mu$ be a Borel probability measure on $[0,1)$  such that for any ``nice" subset $\cF\subseteq [0,1)$ the identity  
$$
\mu (\{x\in [0,1):~(gx \pmod 1) 
 \in \cF\} )=\mu(\cF) 
$$  
holds for $g=2$ and $g=3$, then $\mu$ is Lebesgue measure or some ``trivial" measure.

\section*{Acknowledgement}

The authors are grateful to Trevor Wooley for helpful discussions and patient answering their questions.  
The authors also would like to thank the anonymous referee for the very careful reading of the manuscript and 
many helpful comments. 

This work was  supported   by ARC Grant~DP170100786.

\end{document}